\documentclass{amsart}
\usepackage{amssymb}
\usepackage{amsmath}
\usepackage{amsfonts}
\usepackage{amsxtra}
\usepackage{color}

\theoremstyle{plain}
\newtheorem{theorem}{Theorem}
\newtheorem{proposition}{Proposition}
\newtheorem{lemma}{Lemma}
\newtheorem{step}{Step}

\theoremstyle{remark}
\newtheorem*{notation}{{\sc Notation}}

\theoremstyle{definition}
\newtheorem{example}{Example}

\newcommand{\reft}[1]{Theorem \ref{#1}}
\newcommand{\refp}[1]{Proposition \ref{#1}}
\newcommand{\refl}[1]{Lemma \ref{#1}}
\newcommand{\refs}[1]{Step \ref{#1}}

\newcommand{\refeq}[1]{Eq. (\ref{#1})}

\def\CA{\mathcal{A}}
\def\CB{\mathcal{B}}
\def\CD{\mathcal{D}}
\def\CH{\mathcal{H}}
\def\CG{\mathcal{G}}
\def\CI{\mathcal{I}}
\def\CJ{\mathcal{J}}
\def\CL{\mathcal{L}}
\def\CM{\mathcal{M}}
\def\CO{\mathcal{O}}
\def\CU{\mathcal{U}}
\def\CV{\mathcal{V}}

\def\k{\Bbbk}
\def\C{\mathbb{C}}
\def\N{\mathbb{N}}
\def\R{\mathbb{R}}
\def\S{\mathbb{S}}
\def\qp{\mathbb{Q}_{p}}
\def\fq{\mathbb{F}_{q}}

\newcommand{\map}[3]{#1 \colon #2 \to #3}
\newcommand{\seq}[2]{#1_{1}, \ldots, #1_{#2}}
\newcommand{\set}[2]{\{ #1 \colon #2 \}}
\newcommand{\dset}[2]{\left\{ #1 \colon #2 \right\}}
\def\GL{\operatorname{GL}}
\def\U{\operatorname{U}}
\def\Hom{\operatorname{Hom}}
\def\Res{\operatorname{Res}}
\def\Ind{\operatorname{Ind}}
\def\CInd{\operatorname{c-Ind}}
\def\supp{\operatorname{supp}}
\def\spec{\operatorname{Spec}}
\def\all{\text{ for all }}

\allowdisplaybreaks

\begin{document}


\title[]{Smooth representations of unit groups of split basic algebras over non-Archimedean local fields}

\author[]{Carlos A. M. Andr\'e}

\author[]{Jo\~ ao Dias}

\thanks{This research was made within the activities of the Group for Linear, Algebraic and Combinatorial Structures of the Center for Functional Analysis, Linear Structures and Applications (University of Lisbon, Portugal), and was partially supported by the Portuguese Science Foundation (FCT) through the Strategic Project UID/MAT/04721/2013. The second author was partially supported by the Lisbon Mathematics PhD program (funded by the Portuguese Science Foundation). This work is part of the second author Ph.D. thesis.}

\address[C. A. M. Andr\'e]{Centro de An\'alise Funcional, Estruturas Lineares e Aplica\c c\~oes (Grupo de Estruturas Lineares e Combinat\'orias) \\ Departamento de Matem\'atica \\ Faculdade de Ci\^encias da Universidade de Lisboa \\ Campo Grande, Edif\'\i cio C6, Piso 2 \\ 1749-016 Lisboa \\ Portugal}
\email{caandre@ciencias.ulisboa.pt}

\address[J. Dias]{Centro de An\'alise Funcional, Estruturas Lineares e Aplica\c c\~oes (Grupo de Estruturas Lineares e Combinat\'orias) \\ Departamento de Matem\'atica \\ Faculdade de Ci\^encias da Universidade de Lisboa \\ Campo Grande, Edif\'\i cio C6, Piso 2 \\ 1749-016 Lisboa \\ Portugal}
\email{joaodias104@gmail.com}

\subjclass[2010]{20G25; 22D12; 22D30}

\date{\today}

\keywords{Split basic algebra; unit group; smooth representation; induction with compact support; Gutkin's conjecture}

\begin{abstract}
We consider smooth representations of the unit group $G = \CA^{\times}$ of a finite-dimensional split basic algebra $\CA$ over a non-Archimedean local field. In particular, we prove a version of Gutkin's conjecture, namely, we prove that every irreducible smooth representation of $G$ is compactly induced by a one-dimensional representation of the unit group of some subalgebra of $\CA$. We also discuss admissibility and unitarisability of smooth representations of $G$.
\end{abstract}

\maketitle


Let $\k$ be a non-Archimedean local field (such as a finite extension of the $p$-adic field $\qp$, or a field $\fq((t))$ of formal Laurent series in one variable over a finite field $\fq$); we equip $\k$ with the natural topology making it a locally compact totally disconnected topological field. We fix a non-trivial unitary character $\map{\vartheta}{\k^{+}}{\C^{\times}}$ of the additive group of $\k$, and for each $a \in \k$ we define $\map{\vartheta_{a}}{\k^{+}}{\C^{\times}}$ by $\vartheta_{a}(b) = \vartheta(ab)$ for all $b \in \k$; then, the mapping $a \mapsto \vartheta_{a}$ defines a topological isomorphism between $\k^{+}$ and its Pontryagin dual (see, for example, \cite[Proposition~1.7]{Bushnell2006a}).

Let $\CA$ be a finite-dimensional associative $\k$-algebra (with identity), and let $\CJ(\CA)$ denote the Jacobson radical of $\CA$. We say that $\CA$ is a {\it basic $\k$-algebra} if the $\CJ(\CA)$ equals the set consisting of all nilpotent elements of $\CA$; in \cite{Szegedy1996a}, B. Szegedy refers to $\CA$ as an {\it N-algebra} over $\k$. It follows from Wedderburn's theorem (and from the usual process of ``lifting idempotents'') that a finite-dimensional $\k$-algebra $\CA$ is basic if and only if there are nonzero orthogonal idempotents $\seq{e}{n} \in \CA$ such that $$\CA = \big( \k_{1} e_{1} \oplus \cdots \oplus \k_{n} e_{n} \big) \oplus \CJ(\CA)$$ where $\seq{\k}{n}$ are finite field extensions of $\k$; we refer to $\CD = \k_{1} e_{1} \oplus \cdots \oplus \k_{n} e_{n}$ as the \textit{diagonal subalgebra} of $\CA$. In the case where $\k_{i} = \k$ for all $1 \leq i \leq n$, we refer to $\CA$ as a {\it split basic $\k$-algebra} (in  the terminology of \cite{Szegedy1996a}, $\CA$ is referred to as a {\it DN-algebra} over $\k$). The following easy observation is crucial for inductive arguments; a proof in the case where $\k$ is a finite field can be found in \cite[Lemmas~2.2~and~2.3]{Szegedy1996a}.

\begin{lemma} \label{subalgebra}
Let $\CA$ be a finite-dimensional split basic $\k$-algebra. Then, every subalgebra of $\CA$ is also a split basic $\k$-algebra.
\end{lemma}

\begin{proof}
Let $\CB$ be a subalgebra of $\CA$, and let $\CJ(\CB)$ denote the Jacobson radical of $\CB$. Since $\CJ(\CB) = \CB \cap \CJ(\CA)$, it is clear that $\CB$ is a basic $\k$-algebra; moreover, the $\k$-algebra $\CB\slash \CJ(\CB)$ is naturally isomorphic to the subalgebra $\big(\CB+\CJ(\CA)\big)\slash \CJ(\CA)$ of the semisimple $\k$-algebra $\CA \slash \CJ(\CA)$. Therefore, without loss of generality, we may assume that $\CA$ is a split basic semisimple $\k$-algebra. Since $\CB$ is a basic semisimple $\k$-algebra, there are nonzero orthogonal idempotents $\seq{e'}{m} \in \CB$ such that $\CB = \k_{1} e'_{1} \oplus \cdots \oplus \k_{n} e'_{m}$ where $\seq{\k}{n}$ are finite field extensions of $\k$. On the other hand, let $\seq{e}{n} \in \CA$ be nonzero orthogonal idempotents such that $\CA = \k e_{1} \oplus \cdots \oplus \k e_{n}$. It is straightforward to check that there exists a subset partition $\seq{I}{m}$ of $\{1,\ldots,n\}$ such that $e'_{j} = \sum_{i \in I_{j}} e_{i}$, and thus $e_{j}e'_{i} = e_{i}$ for all $i \in I_{j}$ and all $1 \leq j \leq m$. It follows that $\k_{j}e'_{j}e_{i} = \k_{j}e_{i} \subseteq \CA  e_{i} = \k e_{i}$ for all $i \in I_{j}$ and all $1 \leq j \leq m$, which clearly implies that $\k_{j} = \k$ for all $1 \leq j \leq m$.
\end{proof}

In the following result we list some elementary properties which will be used repeatedly throughout the paper; for a detailed proof (which does not depend on the finiteness of the field $\k$) we refer to \cite[Lemma~2.1]{Halasi2006a}.

\begin{lemma} \label{bimodule}
Let $\CA$ be a finite-dimensional split basic $\k$-algebra, let $\CD$ be the diagonal subalgebra of $\CA$, and let $\seq{e}{n} \in \CA$ be nonzero orthogonal idempotents such that $\CD = \k e_{1} \oplus \cdots \oplus \k e_{n}$. The following properties hold.
\begin{enumerate}
\item If $\CV$ is an (arbitrary) $\CD$-bimodule, then $\CV$ decomposes as a direct sum of the (non-zero) homogeneous sub-bimodules $e_{i}\CV e_{j}$ for $1 \leq i,j \leq n$.
\item For every sub-bimodule $\CV_{1}$ of a $\CD$-bimodule $\CV$, there exists a sub-bimodule $\CV_{2}$ of $\CV$ such that $\CV = \CV_{1} \oplus \CV_{2}$.
\item Every $\CD$-bimodule decomposes as a direct sum of one-dimensional sub-bimodules.
\item If $\CV$ is a one-dimensional $\CD$-bimodule and $v \in \CV$ is such that $\CV = \k v$, then there exist uniquely determined $1 \leq i, j \leq n$ such that $v = e_{i}ve_{j}$.
\end{enumerate}
\end{lemma}

Let $G = \CA^{\times}$ denote the unit group of a split basic $\k$-algebra $\CA$. For any (nilpotent) subalgebra $\CU$ of $\CJ(\CA)$, the set $1+\CU$ is a subgroup of $G$ to which we refer as an \textit{algebra subgroup} of $G$ (as defined in \cite{Isaacs1995a}); similarly, if $\CI \subseteq \CJ(\CA)$ is an ideal of $\CA$, we refer to $1+\CI$ as an \textit{ideal subgroup} of $G$. In the particular case where $\CI = \CJ(\CA)$, it is clear that $P = 1+\CJ(\CA)$ is a normal subgroup of $G$; furthermore, $G$ is the semidirect product $G = TP$ where $T \leq G$ is isomorphic to the unit group of $\CA/\CJ(\CA)$. Since $\CA$ is a split basic $\k$-algebra, $T$ is isomorphic to a direct product $\k^{\times} \times \cdots \times \k^{\times}$ of $n = \dim \CA\slash\CJ(\CA)$ copies of the multiplicative group $\k^{\times}$ of $\k$. In fact, $T = \CD^{\times}$ is the unit group of the diagonal subalgebra $\CD$ of $\CA$; we will refer to $T$ as the {\it diagonal subgroup} of $G$.

\begin{example}
Let $\CM_{n}(\k)$ denote the $\k$-algebra consisting of all $n \times n$ matrices with entries in $\k$, and let $\CA = \CB_{n}(\k)$ denote the Borel subalgebra of $\CM_{n}(\k)$ consisting of all upper-triangular matrices; hence, $G = \CA^{\times}$ is the standard Borel subgroup $B_{n}(\k)$ of the general linear group $\GL_{n}(\k)$ (consisting of all invertible matrices in $\CM_{n}(\k)$). In this case, $T \leq G$ is the standard torus consisting of all diagonal matrices, and $P = 1+\CJ(\CA)$ is the standard unitriangular group; we note that $\CJ(\CA)$ is the nilpotent ideal of $\CA$ consisting of all upper-triangular matrices with zeroes on the main diagonal. In this case, for each $1 \leq i \leq n$, the idempotent $e_{i} \in \CA$ can be chosen to be the elementary matrices $e_{i} = e_{i,i}$ having a unique non-zero entry (equal to $1$) in the position $(i,i)$, and thus $\CD = \k e_{1} \oplus \cdots \oplus \k e_{n}$ is the subalgebra of $\CA$ consisting of all diagonal matrices.
\end{example}

The topology of $\k$ induces naturally a topology in $G = \CA^{\times}$ with respect to which $G$ becomes a locally compact totally disconnected topological group (that is, a topological group such that  every open neighbourhood of the identity contains a compact open subgroup); for simplicity, we follow the terminology of \cite{Boyarchenko2011a} and refer to such a group as an \textit{$\ell$-group}. (For the definition and main properties of $\ell$-groups, we refer mainly to \cite[Chapter~I]{Bushnell2006a}.) We notice that $G$ is second countable. On the other hand, any algebra subgroup $Q$ of $G$ is an \textit{$\ell_{c}$-group} (see \cite{Boyarchenko2011a}) which means that $Q$ is the filtered union of its compact open subgroups; in other words, every element of $Q$ is contained in a compact open subgroup, and any two such subgroups are contained in a third such subgroup.

We recall that a \textit{(complex) representation} of $G$ is a pair $(\pi,V)$ consisting of a complex vector space $V$ (not necessarily of finite dimension) and group homomorphism $\map{\pi}{G}{\GL(V)}$ where $\GL(V)$ denotes the group consisting of all linear automorphisms of $V$; for simplicity, we will refer to $V$ as a \textit{(left) $G$-module} with (linear) $G$-action defined by $gv = \pi(g)v$ for all $g \in G$ and all $v \in V$. A $G$-module $V$ is said to be \textit{smooth} if, for every $v \in V$, the centraliser $G_{v} = \set{g \in G}{gv = v}$ is an open subgroup of $G$; in the particular case where the vector space $V$ is one-dimensional, we naturally obtain a group homomorphism $\map{\vartheta}{G}{\C^{\times}}$ with open kernel. We will refer to such a homomorphism as a \textit{smooth character} of $G$, and denote by $G^{\circ}$ the set consisting of all smooth characters of $G$. A smooth character $\vartheta \in G^{\circ}$ can be naturally viewed as a smooth representation of $G$; indeed, an arbitrary one-dimensional $G$-module is smooth if and only if it affords a smooth character of $G$. Throughout the paper, for every smooth character $\vartheta \in G^{\circ}$, we will denote by $\C_{\vartheta}$ the one-dimensional smooth $G$-module whose underlying vector space is $\C$ and where the (linear) $G$-action is given by $g\alpha = \vartheta(g)\alpha$ for all $g \in G$ and all $\alpha \in \C$. 

It is well-known that $G^{\circ}$ is a group with respect to the usual multiplication of characters; it should not be confused with the \textit{Pontryagin dual} of $G$ which consists of all unitary characters $\map{\vartheta}{G}{\C^{\times}}$ of $G$. By definition, a \textit{unitary character} of $G$ is a continuous group homomorphism $\map{\vartheta}{G}{\C^{\times}}$ whose image $\vartheta(G)$ lies inside the unit circle $\S^{1}$ in $\C$. We observe that every unitary character of $G$ is a smooth character, but the converse is not necessarily true; however, for an arbitrary $\ell_{c}$-group $Q$, it is well-known that $Q^{\circ}$ equals the Pontryagin dual of $Q$ (see \cite[Proposition~1.6]{Bushnell2006a}; see also \cite[Lemma~4.9]{Boyarchenko2011a}).

A morphism between smooth $G$-modules is defined as a morphism of abstract $G$-modules; we refer to such a morphism as a \textit{homomorphism of $G$-modules} (or simply a \textit{$G$-homomorphism}), and denote by $\Hom_{G}(V,V')$ the (complex) vector space consisting of all homomorphisms between $G$-modules $V$ to $V'$. We say that two smooth representations $(\pi,V)$ and $(\pi',V')$ of $G$ are \textit{equivalent} if there is an \textit{isomorphism} $\varphi \in \Hom_{G}(V,V')$ (that is, a linear isomorphism which commutes with the $G$-actions); if this is the case, then the smooth $G$-modules $V$ and $V'$ are said to be \textit{isomorphic}, and we write $V \cong V'$. A smooth $G$-module $V$ is said to be \textit{irreducible} if $V \neq \{0\}$ and $\{0\}$ and $V$ are the only $G$-invariant subspaces of $V$. Since $G$ is an $\ell$-group (hence, $G$ is second countable), there is a natural linear isomorphism $\Hom_{G}(V,V) \cong \C$ for every irreducible smooth $G$-module $V$; the proof of this version of Schur's Lemma is due to H. Jacquet \cite{Jacquet1975a} (it can be found in \cite[pg.~21]{Bushnell2006a}; see also \cite{Cartier1979a}). As a consequence, we deduce that, if $G$ is abelian, then every irreducible smooth $G$-module is one-dimensional (and hence affords a smooth character of $G$).

If $H$ is a subgroup of $G$ and $V$ is a $G$-module, we denote by $V^{H}$ the vector subspace of $V$ consisting of all $v \in V$ such that $hv = v$ for all $h \in H$; then, $V$ is a smooth $G$-module if and only if $V = \bigcup_{K} V^{K}$ where $K$ runs over all compact open subgroups of $G$. A smooth $G$-module $V$ is said to be \textit{admissible} if, for every compact open subgroup $K$ of $G$, the subspace $V^{K}$ is finite-dimensional; on the other hand, $V$ is said to be \textit{unitarisable} if $V$ has a positive definite Hermitian inner product invariant under the action of $G$. We recall the (usual notion) of unitary representations (which are not necessarily smooth) of topological groups. By a \textit{unitary representation} of a topological group $G$ we mean a pair $(\pi,\CH)$ where $\CH$ is an Hilbert space over $\C$ and $\map{\pi}{G}{\U(\CH)}$ is a continuous group homomorphism from $G$ to the group of unitary linear automorphisms of $\CH$ equipped with the strong operator topology; in this case, the representation $(\pi,\CH)$ is said to be \textit{irreducible} if $\CH \neq \{0\}$ and $\{0\}$ and $\CH$ are the only $\pi(G)$-invariant closed subspaces of $\CH$.


A major diference between $\ell$-groups and $\ell_{c}$-groups (hence, in particular, between $G = \CA^{\times}$ and its algebra subgroups) concerns with admissibility and unitarisability of smooth modules. Indeed, \cite[Theorem~1.3]{Boyarchenko2011a} asserts that, if $P$ is an algebra group over a non-Archimedean local field $\k$, then every smooth irreducible $P$-module is admissible and unitarisable; in general, by an algebra group over an arbitrary field $\k$ we mean a group of the form $1+\CJ$ where $\CJ$ is a finite-dimensional nilpotent $\k$-algebra (with product formally defined by $(1+a)(1+b) = 1+a+b+ab$ for all $a,b \in \CJ$). In particular, the unitarisability of a smooth character $\vartheta \in P^{\circ}$ is equivalent to the statement that the image $\vartheta(P)$ lies inside the unit circle $\S^{1}$ in $\C$ (by \cite[Proposition~1.6]{Bushnell2006a}). However, the absolute value $|\cdot|$ of a non-Archimedean local field $\k$ defines a smooth homomorphism $\map{|\cdot|}{\k^{\times}}{\C^{\times}}$ whose image does not lie inside $\S^{1}$.\medskip

The main goal of this paper is to study smooth modules for the unit group $G = \CA^{\times}$ of an arbitrary finite-dimensional split basic $\k$-algebra $\CA$; in particular, we aim to establish that every irreducible smooth $G$-module may be obtained by induction (with compact supports) from a one-dimensional smooth module for the unit group $H = \CB^{\times}$ of some subalgebra $\CB$ of $\CA$.

Let $G$ be an arbitrary $\ell$-group, and let $H$ be a closed subgroup of $G$; hence, $H$ is also an $\ell$-group. Let $W$ be an arbitrary smooth $H$-module, and define $\Ind^{G}_{H}(W)$ to be the (complex) vector space consisting of all functions $\map{\phi}{G}{W}$ which satisfy the following two conditions:
\begin{enumerate}
\item  $\phi(hg) = h\phi(g)$ for all $h \in H$ and all $g \in G$;
\item there is an compact open subgroup $K$ of $G$ such that $\phi(gk) = \phi(g)$ for all $g \in G$ and all $k \in K$.
\end{enumerate}
We define a (linear) $G$-action on $\Ind^{G}_{H}(W)$ by the rule $$(g\phi)(g') = \phi(g'g),\qquad g,g' \in G,\ \phi \in\Ind^{G}_{H}(W);$$ then, $\Ind^{G}_{H}(W)$ becomes a smooth $G$-module, to which we refer as the smooth $G$-module \textit{smoothly induced} by $W$. If $\map{\rho}{H}{\GL(W)}$ is the smooth representation of $H$ which is canonically determined by the smooth $H$-module $W$, then we will denote by $\Ind^{G}_{H}(\rho)$ the smooth representation of $G$ which is canonically determined by the smooth $G$-module $\Ind^{G}_{H}(W)$. 
In particular, for every smooth character $\vartheta \in H^{\circ}$, we obtain a smooth representation $\Ind^{G}_{H}(\vartheta)$ of $G$ on the vector space $\Ind^{G}_{H}(\C_{\vartheta})$ consisting of all complex-valued functions $\map{\phi}{G}{\C}$ which satisfy the two conditions as above; notice that condition (i) reads as $\phi(hg) = \vartheta(h)\phi(g)$ for all $h \in H$ and all $g \in G$.

On the other hand, we define $\CInd^{G}_{H}(W)$ to be the vector subspace of $\Ind^{G}_{H}(W)$ consisting of all functions $\map{\phi}{G}{W}$ which are \textit{compactly supported modulo $H$}, which means that $\supp(\phi) \subseteq HC$ for some compact subset $C \subseteq G$; as usual, $\supp(\phi)$ denotes the \textit{support} of $\phi$. 
It is clear that $\CInd^{G}_{H}(W)$ is a $G$-invariant vector subspace of $\Ind^{G}_{H}(W)$, and thus $\CInd^{G}_{H}(W)$ becomes a smooth $G$-module to which we refer as the smooth $G$-module \textit{compactly induced} (or simply \textit{c-induced}) by $W$. As in the case of smooth induction, if $\map{\rho}{H}{\GL(W)}$ is the smooth representation of $H$ which is canonically determined by the smooth $H$-module $W$, then we will denote by $\CInd^{G}_{H}(\rho)$ the smooth representation of $G$ which is canonically determined by the smooth $G$-module $\CInd^{G}_{H}(W)$. 
%
In the general situation, it is obvious that $\CInd^{G}_{H}(W) = \Ind^{G}_{H}(W)$ whenever the coset space $H\backslash G$ is compact; however, the canonical inclusion map $\CInd^{G}_{H}(W) \hookrightarrow \Ind^{G}_{H}(W)$ may be an isomorphism even when $H\backslash G$ is not compact (this may occur, for example, in the case where $G$ is an algebra group and $H$ is an algebra subgroup of $G$; see \cite[Theorem~1.3]{Boyarchenko2011a}) .

%
%

The main goal of this paper is to prove the following result.

\begin{theorem} \label{main}
Let $\CA$ be a finite-dimensional split basic algebra over a non-Archime\-dean local field $\k$, let $G = \CA^{\times}$ be the unit group of $\CA$, and  let $V$ be an irreducible smooth $G$-module. Then, there exists a subalgebra $\CB$ of $\CA$ and a smooth character of the unit group $H = \CB^{\times}$ such that $V \cong \CInd^{G}_{H}(\C_{\vartheta})$.
\end{theorem}

We should mention that the analogous result has been proved by Z. Halasi in the case where $\k$ is a finite field (see \cite[Theorem~1.3]{Halasi2006a}), and this result extends a previous result (see \cite[Theorem~1.2]{Halasi2004a}) for arbitrary algebra groups over finite fields. More generally, E. Gutkin in \cite{Gutkin1974a} claimed that every unitary irreducible representation of an algebra group over a locally compact self-dual field is induced (in the sense of unitary representations) by a unitary character of some algebra subgroup; this statement obviously includes the case of an algebra group over a finite field, and the proof of it was shown to be defective by I.M. Isaacs in \cite[Section~10]{Isaacs1995a}. However, the theorem by Z. Halasi has been successfully generalised to the general situation by M. Boyarchenko in the paper \cite{Boyarchenko2011a}; in particular, \cite[Theorem~1.3]{Boyarchenko2011a}) asserts that every irreducible smooth representation of an algebra group over any non-Archimedean local field $\k$ is admissible and unitarisable. In general, this is not true for the unit group of an arbitrary split basic $\k$-algebra. By the way of example, the multiplicative group $\k^{\times}$ of a non-Archimedean local field has smooth characters which are not unitarisable; as we mentioned above, the absolute value of $\k$ defines a smooth character of $\k^{\times}$ whose image does not lie in the unit circle $\S^{1}$. On the other hand, the following example shows that there are irreducible smooth representations which are not admissible.

\begin{example} \label{example}
Let $\k$ be an arbitrary non-Archimedean local field, and let $\CA = \CB_{2}(\k)$ be the standard Borel algebra of $\CM_{2}(\k)$; hence, $G = \CA^{\times}$ is the standard Borel subgroup $B_{2}(\k)$ of $\GL_{2}(\k)$ consisting of all upper-triangular matrices. It is obvious that $G$ is the semidirect product $G = TP$ where $T \cong \k^{\times} \times \k^{\times}$ is the subgroup of $G$ consisting of all diagonal matrices and $P \cong \k^{+}$ is the abelian normal subgroup of $G$ consisting of all unipotent matrices.

We consider the conjugation action of $G$ on $P^{\circ}$: for every $g \in G$ and every $\vartheta \in P^{\circ}$, we define $\vartheta^{g} \in P^{\circ}$ by $\vartheta^{g}(x) = \vartheta(gxg^{-1})$ for all $x \in P$. It is clear that there are exactly two $G$-orbits on $P^{\circ}$, namely: the singleton $\{1_{P}\}$ consisting of the trivial character of $P$, and its complement $P^{\circ} \setminus \{1_{P}\}$. Both these $G$-orbits are locally closed, and thus by \cite[Corollaire~2 au Th\'eor\`eme~3]{Rodier1977a} there are two distinct families of irreducible smooth $G$-modules each one corresponding to one of these two $G$-orbits. On the one hand, one has the family consisting of one-dimensional $G$-modules corresponding to the smooth characters $\vartheta \in G^{\circ}$ which satisfy $P \subseteq \ker(\vartheta)$; indeed, $P$ lies is the kernel of every smooth character of $G$.

On the other hand, there is a family corresponding to a fixed non-trivial smooth character $\vartheta \in P^{\circ}$. In this case, the centraliser\footnote{Throughout the paper, whenever a group $G$ acts on a set $\Omega$, we denote by $G_{\omega}$ the centraliser of $\omega \in \Omega$ in $G$.} $G_{\vartheta}$ is the (internal) direct product $G_{\vartheta} = ZP$ where $Z = Z(G)$ is the center of $G$; notice that $Z \cong \k^{\times}$, whereas $P \cong \k^{+}$ (hence, $G_{\vartheta} \cong \k^{\times} \times \k^{+}$). Since $G_{\vartheta}$ is abelian, every irreducible smooth $G_{\vartheta}$-module is one-dimensional. By Rodier's result, for every smooth character $\tau \in (G_{\vartheta})^{\circ}$ satisfying $\tau_{P} = \vartheta$, the c-induced smooth $G_{\vartheta}$-module $\CInd^{G}_{G_{\vartheta}}(\C_{\tau})$ is irreducible (and clearly of dimension $\geq 2$); moreover, the mapping $\tau \mapsto \CInd_{G_{\vartheta}}^{G}(\C_{\tau})$ defines a one-to-one correspondence between smooth characters of $G_{\vartheta}$ satisfying $\tau_{P} = \vartheta$ and irreducible smooth $G$-modules with dimension $\geq 2$.

Now, let $\CO$ denote the ring of algebraic integers of $\k$, let $$T_{0} = \dset{\begin{bmatrix} \alpha & 0 \\ 0 & \beta \end{bmatrix}}{\alpha,\beta \in \CO^{\times}}$$ where $\CO^{\times}$ denote the unit group of $\CO$, and consider the subgroup $G_{0} = T_{0}G_{\vartheta}$ of $G$; notice that $G_{0}$ is an open (hence, also a closed) subgroup of $G$. Let $\tau \in (G_{\vartheta})^{\circ}$ be an arbitrary smooth character of $G_{\vartheta}$ satisfying $\tau_{P} = \vartheta$, and consider the c-induced $G$-module $V = \CInd_{G_{\vartheta}}^{G}(\C_{\tau})$. By the transitivity of c-induction (see \cite[Proposition~2.25(b)]{Bernstein1976a}), we have $V = \CInd_{G_{0}}^{G}(V_{0})$ where $V_{0} = \CInd^{G_{0}}_{G_{\vartheta}}(\C_{\tau})$. Since $V$ is irreducible, it is obvious that $V_{0}$ is an irreducible smooth $G_{0}$-module. Since $G_{0}$ is an open subgroup of $G$, it follows from \cite[Lemma~2.5]{Bushnell2006a} that $V_{0}$ is naturally embedded as a vector subspace of $V$ and that we have a direct sum decomposition $$V = \bigoplus_{g \in \CG} gV_{0}$$ where $\CG$ is a complete set of representatives of the coset space $G\slash G_{0}$.

Let $K$ be a sufficiently small open compact subgroup of $G_{0}$ such that $V_{0}^{\;K} \neq \{0\}$ (hence, $V^{K} \neq \{0\}$), and let $$g_{m} = \begin{bmatrix} \varpi^{m} & 0 \\ 0 & \varpi^{-m} \end{bmatrix},\qquad m \in \N,$$ where $\varpi \in \CO$ is the uniformiser of $\k$ (so that $\varpi\CO$ is the unique maximal ideal of $\CO$). For every $m \in \N$, we have $g_{m} \notin G_{0}$ and $g_{m}Kg_{m}^{-1} \subseteq K$. It follows that $(g_{m}V_{0})^{K} = g_{m} V_{0}^{\;K} \neq \{0\}$, and thus $$V^{K} = \bigoplus_{g \in \CG} (gV_{0})^{K}$$ has infinite dimension. Therefore, the smooth $G$-module $V$ is not admissible, and hence it follows from Rodier's theorem that an irreducible smooth $G$-module is admissible if and only if it is one-dimensional.
\end{example}

The proof of \reft{main} relies on a refinement of the general techniques used by M. Boyarchenko in the paper \cite{Boyarchenko2011a}.

\begin{notation}
Henceforth, we fix the following notation which we will use repeatedly, without always recalling their meaning.
\begin{itemize}
\item $\k$ is a non-Archimedean local field.
\item $\CA$ is a finite-dimensional split basic $\k$-algebra.
\item $G = G(\CA)$ is the unit group of $\CA$.
\item $\CJ = \CJ(\CA)$ is the Jacobson radical of $\CA$, and $P = 1+\CJ$;
\item $\CD$ is the diagonal subalgebra of $\CA$, and $T = \CD^{\times}$ is the diagonal subgroup of $G$.
\end{itemize}
\end{notation}

We recall that $G$ is a second countable $\ell$-group, and that $P$ is a normal $\ell_{c}$-subgroup of $G$; moreover, $G$ is the semidirect product $G = TP$. We also recall that by an algebra subgroup of $G$ we mean a subgroup of $G$ of the form $Q = 1+\CJ$ where $\CJ$ is a (nilpotent) subalgebra of $\CJ(\CA)$ (hence, $Q \subseteq P$), whereas an ideal subgroup of $G$ is an algebra subgroup $Q = 1+\CJ$ where $\CJ \subseteq \CJ(\CA)$ is an ideal of $\CA$ (in particular, every ideal subgroup of $G$ is a normal subgroup of $G$).

Let $Q$ be an arbitrary algebra subgroup of $G$ (hence, $Q$ is a subgroup of $P$), and let $W$ be an arbitrary smooth $Q$-module. For every smooth character $\vartheta \in Q^{\circ}$, we denote by $W(\vartheta)$ the vector subspace of $W$ linearly spanned by vectors $xw-\vartheta(x)w$ for $x \in Q$ and $w \in W$, and consider the quotient $W_{\vartheta} = W\slash W(\vartheta)$; notice that $W_{\vartheta}$ is the largest quotient of $W$ where $Q$ acts via the character $\vartheta$. We define the \textit{spectral support} of $W$ to be the subset $$\spec_{Q}(W) = \set{\vartheta \in Q^{\circ}}{W_{\vartheta} \neq \{0\}}$$ of $Q^{\circ}$; we recall that, since $Q$ is an $\ell_{c}$-group, $Q^{\circ}$ equals the Pontryagin dual of $Q$. If $V$ is an arbitrary smooth $G$-module, then $V$ is also a smooth $Q$-module, and thus we may define the quotient $V_{\vartheta} = V\slash V(\vartheta)$ as above; in this case, we refer to $\spec_{Q}(V)$ as the \textit{spectral support of $V$ with respect to $Q$}. In the case where $Q$ is an ideal subgroup of $G$, then $Q$ is a normal subgroup, and hence $G$ acts by conjugation on $Q^{\circ}$: for every $g \in G$ and every $\vartheta \in Q^{\circ}$, we define $\vartheta^{g} \in Q^{\circ}$ by $\vartheta^{g}(x) = \vartheta(gxg^{-1})$ for all $x \in Q$. We also observe that $V_{\vartheta}$ is a smooth $G_{\vartheta}$-module which satisfies $$x\overline{v} = \vartheta(x)\overline{v},\qquad x \in Q,\ \overline{v} \in V_{\vartheta}$$ (that is, the restriction $\Res^{G_{\vartheta}}_{Q}(V_{\vartheta})$ of $V_{\vartheta}$ to $Q$ is \textit{isotypic} of type $\vartheta$). In this situation, the following auxiliary result will be important for us; for any subgroup $H$ of $G$, we will denote by $\overline{[H,H]}$ the closure of the commutator subgroup $[H,H]$ of $H$.

\begin{lemma} \label{spec1}
Let $V$ be an arbitrary smooth $G$-module, and let $Q$ be an ideal subgroup of $G$. Then, $\spec_{Q}(V)$ is a $G$-invariant subset of $Q^{\circ}$ and $$V_{0} = \bigcap_{\vartheta \in \spec_{Q}(V)} V(\vartheta)$$ is a $G$-submodule of $V$ (that is, a $G$-invariant vector subspace of $V$). In particular, if $V$ is irreducible and $\spec_{Q}(V)$ is non-empty, then $\overline{[Q,Q]}$ acts trivially on $V$, and hence $V$ becomes naturally as an irreducible smooth $\big( G \slash \overline{[Q,Q]}\big)$-module.
\end{lemma}

\begin{proof}
For the first assertion, it is enough to observe that $V(\vartheta^{g}) = g V(\vartheta)$ for all $\vartheta \in Q^{\circ}$ and all $g \in G$. For the second assertion, we start by observing that $V_{0} \neq V$ (otherwise, $V(\vartheta) = V$ for all $\vartheta \in \spec_{Q}(V)$), and thus $V_{0} = \{0\}$ (because $V$ is irreducible). It follows that the natural linear map $$V \longrightarrow \prod_{\vartheta \in \spec_{Q}(V)} V_{\vartheta}$$ is injective, and thus $\overline{[Q,Q]}$ acts trivially on $V$ (because $\overline{[Q,Q]} \subseteq \ker(\vartheta)$ for all $\vartheta \in Q^{\circ}$).
\end{proof}

As a consequence of \cite[Corollaire~1 au Th\'eor\`eme~3]{Rodier1977a}, we conclude that $\spec_{Q}(V) = \vartheta^{G}$ whenever $Q$ be an ideal subgroup of $G$ and $V$ is an irreducible smooth $G$-module such that $\spec_{Q}(V) \neq \emptyset$; indeed, we have the following result.

\begin{lemma} \label{orbit}
 Let $Q$ be an ideal subgroup of $G$, and let $\vartheta\in Q^{\circ}$. Then, the $G$-orbit $\vartheta^{G} = \set{\vartheta^{g}}{g \in G}$ is a locally closed subset of $Q^{\circ}$.
\end{lemma}

\begin{proof}
Since $G$ is a second countable locally compact group, it is $\sigma$-compact (that is, a union of countably many compact subspaces), and thus $\vartheta^{G}$ is naturally homeomorphic to the quotient space $G\slash G_{\vartheta}$ (see \cite[Proposition~2.44]{Folland2016a}). Since $G\slash G_{\vartheta}$ is a locally compact space, we conclude that $\vartheta^{G}$ is also locally compact, and thus it is open in its closure $\overline{\vartheta^{G}}$ in $Q^{\circ}$.
\end{proof}

\begin{lemma} \label{orbit2}
Let $V$ be an irreducible smooth $G$-module, let $Q$ be an ideal subgroup of $G$, and let $\vartheta\in Q^{\circ}$ be such that $V_{\vartheta} \neq \{0\}$. Then, $\spec_{Q}(V) = \vartheta^{G}$; moreover, $V_{\vartheta}$ is an irreducible smooth $G_{\vartheta}$-module and $V \cong \CInd^{G}_{G_{\vartheta}}(V_{\vartheta})$.
\end{lemma}

\begin{proof}
We consider the closure $\overline{[Q,Q]}$ of the commutator subgroup of $Q$ and the quotient group $\overline{G} = G \slash \overline{[Q,Q]}$; then, $\overline{G}$ is an $\ell$-group, and $\overline{Q} = Q\slash \overline{[Q,Q]}$ is an abelian normal subgroup of $\overline{G}$. By \refl{spec1}, $V$ is an irreducible smooth $\overline{G}$-module; moreover, \refl{orbit} clearly implies that every $\overline{G}$-orbit is a locally closed subset of $\overline{Q}^{\,\circ} \cong Q^{\circ}$. Therefore, it follows from \cite[Corollaire~1 au Th\'eor\`eme~3]{Rodier1977a} that $\spec_{\overline{Q}}(V) = \vartheta^{\overline{G}}$ and that $V_{\vartheta}$ is an irreducible smooth $\overline{G}_{\vartheta}$-module; moreover, \cite[Corollaire~2 au Th\'eor\`eme~3]{Rodier1977a} implies that $V \cong \CInd^{\overline{G}}_{\overline{G}_{\vartheta}}(V_{\vartheta})$. The result follows because $\vartheta^{G} = \vartheta^{\overline{G}}$, $\overline{[Q,Q]} \subseteq G_{\vartheta}$ and $\overline{G}_{\vartheta} = G_{\vartheta}\slash \overline{[Q,Q]}$.
\end{proof}

Our next aim is to prove that, for every irreducible smooth $G$-module $V$, there exists an ideal subgroup $Q$ of $G$ (depending on $V$) such that $\spec_{Q}(V) \neq \emptyset$; the existence of the ideal subgroup $Q$ will be established using a slightly modified key construction due to M. Boyarchenko (see \cite[Section~5.2]{Boyarchenko2011a}). We start by proving the following auxiliary result.

\begin{lemma} \label{restriction}
Let $V$ be an irreducible smooth $G$-module, and let $Q$ be an arbitrary algebra subgroup of $G$. Then, the smooth $Q$-module $\Res^{G}_{Q}(V)$ has an irreducible quotient.
\end{lemma}

\begin{proof}
Since $Q$ is an $\ell_{c}$-group, every irreducible smooth $Q$-module is admissible (by \cite[Theorem~1.3]{Boyarchenko2011a}), and thus $\Res^{G}_{Q}(V)$ has an irreducible quotient (by \cite[Corollary~4.8]{Boyarchenko2011a}).
\end{proof}

Next, we deal with the case where $\Res^{G}_{P}(V)$ has a one-dimensional quotient.

\begin{lemma} \label{linear1}
Let $V$ be an irreducible smooth $G$-module, and let $W$ be an irreducible quotient of $\Res^{G}_{P}(V)$. Suppose that $W$ is one-dimensional, and let $\vartheta \in P^{\circ}$ be the character afforded by $W$. Then, $\vartheta$ is $G$-invariant if and only if $V$ is one-dimensional.
\end{lemma}

\begin{proof}
It is clear that $\vartheta$ is $G$-invariant whenever $V$ is one-dimensional. Conversely, suppose that $\vartheta$ is $G$-invariant. Then, $V(\vartheta)$ is a $G$-invariant vector subspace of $V$, and hence either $V(\vartheta) = \{0\}$ or $V(\vartheta) = V$ (because $V$ is irreducible). Since $P$ is an $\ell_{c}$-group, \cite[Proposition~2.35]{Bernstein1976a} implies that $W_{\vartheta} = W\slash W(\vartheta)$ is an epimorphic image of $V_{\vartheta} = V\slash V(\vartheta)$, and thus $V_{\vartheta} \neq 0$ (because $W(\vartheta) = \{0\}$, and hence $W \cong W_{\vartheta}$). Therefore, we must have $V(\vartheta) = \{0\}$, and thus $$xv = \vartheta(x)v,\qquad x \in P,\ v \in V.$$ It follows that the restriction $\Res^{G}_{T}(V)$ of $V$ to the diagonal subgroup $T$ of $G$ is irreducible; indeed, if $V'$ is $T$-submodule of $\Res^{G}_{T}(V)$, then $V'$ is also a $G$-submodule of $V$ (because every vector subspace of $V$ is $P$-invariant), and so either $V' = \{0\}$ or $V' = V$. Since $T$ is an abelian $\ell$-group, Schur's lemma implies that $V$ is one-dimensional, and this completes the proof.
\end{proof}

In the following, we let $V$ be an irreducible smooth $G$-module, and assume that $\dim V \geq 2$. Let $W$ be an irreducible quotient of $\Res^{G}_{P}(V)$. On the one hand, suppose that $W$ is one-dimensional, and let $\vartheta \in P^{\circ}$ be the character afforded by $W$. Since $\dim V \geq 2$, we have $G_{\vartheta} \neq G$ (by the previous lemma); moreover, by \refl{orbit2}, $V_{\vartheta}$ is an irreducible smooth $G_{\vartheta}$-module and $V \cong \CInd^{G}_{G_{\vartheta}}(V_{\vartheta})$. As in the proof of \refl{linear1}, we conclude that $V_{\vartheta}$ is one-dimensional; notice that $W_{\vartheta}$ is a one-dimensional irreducible quotient of $\Res^{G_{\vartheta}}_{P}(V_{\vartheta})$. Therefore, \reft{main} holds in this situation once we prove that $G_{\vartheta}$ is the unit group of some subalgebra of $\CA$; we observe that $G_{\vartheta}$ is the semidirect product $G_{\vartheta} = T_{\vartheta}P$.

\begin{proposition} \label{centraliser1}
For every ideal subgroup $Q$ of $G$ and every $\vartheta \in Q^{\circ}$, the centraliser $T_{\vartheta}$ is the unit group of a subalgebra of $\CD$.
\end{proposition}

\begin{proof}
If $\CJ \subseteq \CJ(\CA)$ is an ideal of $\CA$ and $\vartheta \in (1+\CJ)^{\circ}$, then it is straightforward to check that $$\CD_{\vartheta} =  \set{d \in \CD}{\vartheta(1+ad) = \vartheta(1+da) \all a \in \CJ}$$ is a subalgebra of $\CD$ with $(\CD_{\vartheta})^{\times} = T_{\vartheta}$. By the way of example, let $d,d' \in \CD_{\vartheta}$, and let $a \in \CJ$ be arbitrary. Then, since $a(1+da)^{-1} = (1+ad)^{-1}a$, we deduce that
\begin{align*}
\vartheta(1+da+d'a) &= \vartheta(1 + d'a(1+da)^{-1}) \vartheta(1+da) \\ &= \vartheta(1 + a(1+da)^{-1}d') \vartheta(1+ad) \\ &= \vartheta(1+ad) \vartheta(1 + (1+ad)^{-1}ad') \\ &= \vartheta(1+ad+ad'),
\end{align*}
and thus $d+d' \in \CD_{\vartheta}$.
%
\end{proof}

As we mentioned above, this completes the proof of the following particular case of \reft{main}.

\begin{proposition} \label{main1}
Let $V$ be an irreducible smooth $G$-module, and let $W$ be an irreducible quotient of $\Res^{G}_{P}(V)$. Suppose that $W$ is one-dimensional, and let $\vartheta \in P^{\circ}$ be the character afforded by $W$. Then, $G_{\vartheta}$ is the unit group of some subalgebra of $\CA$, and $V_{\vartheta}$ is a one-dimensional smooth $G_{\vartheta}$-module such that $V \cong \CInd^{G}_{G_{\vartheta}}(V_{\vartheta})$.
\end{proposition}

\begin{proof}
We have already proved that the smooth $G_{\vartheta}$-module $V_{\vartheta}$ is one-dimensio\-nal and that $V = \CInd^{G}_{G_{\vartheta}}(V_{\vartheta})$. If $\CD$ is the diagonal subalgebra of $\CA$ and $T = \CD^{\times}$, then the previous proposition assures that $T_{\vartheta}$ is the unit group of some subalgebra $\CD_{\vartheta}$ of $\CD$. Since $G_{\vartheta} = T_{\vartheta}P$ and since $\CJ(\CA)$ is an ideal of $\CA$, it follows that $\CA_{\vartheta} = \CD_{\vartheta} \oplus \CJ(\CA)$ is a subalgebra of $\CA$ and that $G_{\vartheta} = (\CA_{\vartheta})^{\times}$ is the unit group of $\CA_{\vartheta}$.
\end{proof}

For the proof of \reft{main} will be proceed by induction on $\dim \CA$. By the results above, the inductive step depends on the existence of an ideal subgroup $Q$ of $G$ such that $\spec_{Q}(V) \neq \emptyset$ where $V$ is an arbitrary irreducible smooth $G$-module; moreover, we are reduced to the case where $\dim W \geq 2$ for every irreducible quotient $W$ of $\Res^{G}_{P}(V)$ (which obviously implies that $\dim V \geq 2$).\medskip 

%

In what follows, we fix the following notation which we will use repeatedly in the subsequent results (without always recalling their meaning). Let $n \geq 2$ be an integer such that $\CJ^{n} \neq \{0\}$, and write $N = 1+\CJ^{n}$. Since $\CJ^{n} \subseteq \CJ^{n-1}$ are ideals of $\CA$, it follows from \refl{bimodule} that there exists an ideal $\CL$ of $\CA$ such that $\CJ^{n} \subseteq \CL \subseteq \CJ^{n-1}$ and $\dim \CL = \dim \CJ^{n} + 1$. We fix such an ideal, and set $Q = 1+\CL$.  [As usual, if $g$ and $h$ are elements of a group, then we define $g^{h} = h^{-1} gh$ and $[g,h] = g^{-1} h^{-1} gh = g^{-1} g^{h}$.]

\begin{lemma} \label{vphi}
Let $\varsigma \in N^{\circ}$ be $P$-invariant, and define $$\CJ_{\varsigma} = \set{a \in \CJ}{\varsigma([1+a,1+u]) = 1 \all u \in \CL}.$$ Then, $\CJ_{\varsigma}$ is a subalgebra of $\CJ$ satisfying $\CJ^{2} \subseteq \CJ_{\varsigma}$ and $\dim \CJ_{\varsigma} \geq \dim \CJ - 1$. Furthermore, if we define the map $\map{\varphi_{\varsigma}}{P}{Q^{\circ}}$ by the rule $$\varphi_{\varsigma}(g)(h) = \varsigma\big([g,h]\big),\qquad g \in P,\ h \in Q,$$ then $\varphi_{\varsigma}$ is a group homomorphism with $\ker(\varphi_{\varsigma}) = 1+\CJ_{\varsigma}$ and $\varphi_{\varsigma}(P) \subseteq N^{\perp}$ where $N^{\perp} = \set{\tau \in Q^{\circ}}{N \subseteq \ker(\tau)}$ is the orthogonal subgroup of $N$ in $Q^{\circ}$; hence, $\varphi_{\varsigma}$ defines naturally a group homomorphism $\map{\overline{\varphi}_{\varsigma}}{P}{(Q\slash N)^{\circ}}$.
\end{lemma}

\begin{proof}
We first observe that the map $\varphi_{\varsigma}$ is a well-defined group homomorphism. On the one hand, we have $[P,Q] \subseteq [1+\CJ,1+\CJ^{n-1}] \subseteq 1+\CJ^{n} = N$. On the other hand, since $[g,hk] = [g,k] [g,h]^{k}$, we deduce that $$\varphi_{\varsigma}(g)(hk) = \varsigma([g,k]) \varsigma([g,h]) = \varphi_{\varsigma}(g)(h) \varphi_{\varsigma}(g)(k)$$ for all $g \in P$ and all $h,k \in Q$; we recall that $\varsigma$ is $P$-invariant. It follows that, for every $g \in P$, the map $\map{\varphi_{\varsigma}(g)}{Q}{\C^{\times}}$ is indeed a (smooth) character of $Q$. Similarly, since $[gh,k] = [g,k]^{h} [h,k]$, we have $$\varphi_{\varsigma}(gh)(k) = \varsigma([g,k]) \varsigma([h,k]) = \varphi_{\varsigma}(g)(k) \varphi_{\varsigma}(h)(k)$$ for all $g,h \in P$ and all $k \in Q$, and so $\varphi_{\varsigma}$ is a group homomorphism.

Now, since $[P,N] \subseteq \ker(\varsigma)$ (because $\varsigma$ is $P$-invariant), the image $\varphi_{\varsigma}(P)$ clearly lies in $N^{\perp}$; moreover, it is obvious (by the definition) that $\ker(\varphi_{\varsigma}) = 1+\CJ_{\varsigma}$. Let $a \in \CJ$ and $\alpha \in \k$ be arbitrary. It is straightforward to check that $$[1+\alpha a, 1+u][1+a,1+\alpha u]^{-1} \in 1+\CJ^{n+1}$$ for all $u \in \CJ^{n-1}$; indeed, as in the proof of \cite[Proposition~3.1(b) (pg. 547)]{Boyarchenko2011a}, we deduce that $$[1+\alpha a, 1+u]\,[1+a,1+\alpha u]^{-1} \in [1+\CJ, 1+\CJ^{n}]$$ for all $u \in \CJ^{n-1}$. Since $\varsigma$ is $P$-invariant (and $[1+\CJ, 1+\CJ^{n}] = [P,N]$), we conclude that
\begin{equation} \label{scalar}
\varsigma([1+\alpha a, 1+u]) = \varsigma([1+a,1+\alpha u]) \tag{$\dagger$}
\end{equation}
for all $u \in \CJ^{n-1}$, and this clearly implies that $\alpha u \in \CJ_{\varsigma}$ for all $\alpha \in \k$ and all $u \in \CJ_{\varsigma}$. On the other hand,\cite[Theorem~1.4]{Halasi2004a} implies that $$[1+\CJ^{2},1+\CL] \subseteq [1+\CJ^{2},1+\CJ^{n-1}] \subseteq [1+\CJ,1+\CJ^{n}] = [P,N],$$ and thus $\CJ^{2} \in \CJ_{\varsigma}$. Since $\ker(\varphi_{\varsigma}) = 1+\CJ_{\varsigma}$ and since $$(1+u+v)^{-1} (1+u)(1+v) = 1+(1+u+v)^{-1} uv \in 1+\CJ^{2},$$ we see that $u+v \in \ker(\varphi_{\varsigma})$ for all $u,v \in \CJ_{\varsigma}$. It follows that $\CJ_{\varsigma}$ is an ideal of $\CJ$ with $\CJ^{2} \subseteq \CJ_{\varsigma}$.

Finally, notice that $N^{\perp} \cong (Q\slash N)^{\circ}$ and that $$Q\slash N = (1+\CL)\slash(1+\CJ^{n}) \cong 1+(\CL\slash\CJ^{n}) \cong \k^{+}.$$ Therefore, since $P\slash \ker(\varphi_{\varsigma}) \cong \varphi_{\varsigma}(P) \subseteq N^{\perp}$, we conclude that $\dim \CJ - \dim \CJ_{\varsigma} \leq 1$, and this completes the proof.
\end{proof}

The following result is essentially a particular case of \cite[Lemma~3.4]{Halasi2006a}; for convenience of the reader, we include a proof which avoids the finiteness of the base field $\k$.

\begin{lemma} \label{ideal}
Let $\varsigma \in N^{\circ}$ be $G$-invariant, let $\CI$ be an ideal of $\CA$ with $\CJ^{2} \subseteq \CI \subseteq \CJ$, and let $\CJ_{\varsigma} \subseteq \CJ$ be defined as in \refl{vphi}. Then, $\CI_{\varsigma} = \CI \cap \CJ_{\varsigma}$ is an ideal of $\CA$.
\end{lemma}

\begin{proof}
The result is obvious in the case where $\CJ_{\varsigma} = \CJ$; hence, we assume that $\CJ_{\varsigma} \neq \CJ$, so that $\dim\CJ_{\varsigma} = \dim\CJ -1$ (by the previous lemma). The result is also clearly true in the case where $\CI \subseteq \CJ_{\varsigma}$; thus, we may assume that $\CJ_{\varsigma}+\CI = \CJ$, which implies that $\dim\CI_{\varsigma} = \dim\CI -1$. We now proceed by induction on $\dim \CI$, the result being obvious if $\dim \CI = \dim \CJ^{2} + 1$; indeed, since $\CJ^{2} \subseteq \CI \cap \CJ_{\varsigma} \subseteq \CI$, either $\CI \cap\CJ_{\varsigma} = \CJ^{2}$, or $\CI \cap\CJ_{\varsigma} = \CI$. Therefore, we may assume that $\dim \CI \geq \dim \CJ^{2} + 2$, and that the result is true whenever $\CI'$ is a ideal of $\CA$ with $\CJ^{2} \subseteq \CI' \subseteq \CJ$ and $\dim \CI' < \dim \CI$.

%

Let $\CI'_{\varsigma}$ be the unique ideal of $\CA$ which is maximal with respect to the condition $\CI'_{\varsigma} \subseteq \CI_{\varsigma}$; hence, we must prove that $\CI'_{\varsigma} = \CI_{\varsigma}$. 
Since $\CI'_{\varsigma}$ is clearly a $\CD$-bimodule, \refl{bimodule} assures that $\CI = \CI'_{\varsigma} \oplus \CV$ for some sub-bimodule $\CV$ of $\CI$. 
Let $\CV_{\varsigma} = \CV \cap \CJ_{\varsigma}$, and note that $\CI_{\varsigma} = \CI'_{\varsigma} \oplus \CV_{\varsigma}$; hence, $\CI'_{\varsigma} = \CI_{\varsigma}$ if and only if $\CV_{\varsigma} = \{0\}$. By the way of contradiction, we assume that $\CV_{\varsigma} \neq \{0\}$; notice that $\CV_{\varsigma} \neq \CV$ (otherwise, $\CI_{\varsigma} = \CI$). Since $Q$ a $T$-invariant subgroup of $G$ (because it is an ideal subgroup of $G$), we deduce that $$\varsigma([1+t^{-1} at,h]) = \varsigma([1+a,tht^{-1}]^{t}) = \varsigma([1+a,tht^{-1}]) = 1$$ for all $a \in \CV_{\varsigma}$, all $t \in T$, and all $h \in Q$ Since $\CV_{\varsigma} \subseteq \CJ_{\varsigma}$ (and since $\CV$ is clearly $T$-invariant), it follows that $\CV_{\varsigma}$ is a $T$-invariant vector subspace of $\CV$.

On the other hand, let $\CV' \neq \{0\}$ be a proper sub-bimodule of $\CV$, and let $\CI' = \CI'_{\varsigma}+\CV'$. Then, $\CI'$ is an ideal of $\CA$ with $\CI' \subsetneq \CI$, and thus $\CI' \cap \CJ_{\varsigma}$ is an ideal of $\CA$ (by the inductive hypothesis). Since $\CI'_{\varsigma} \subseteq \CI' \cap \CJ_{\varsigma} \subseteq \CI \cap \CJ_{\varsigma} = \CI_{\varsigma}$, we conclude that $\CI' \cap \CJ_{\varsigma} = \CI'_{\varsigma}$ (by the maximality of $\CI'_{\varsigma}$), and thus $$\CV_{\varsigma} \cap \CV' = (\CJ_{\varsigma} \cap \CI') \cap \CV = \CI'_{\varsigma} \cap \CV = \{0\}.$$ Therefore, the vector subspaces $\CV$ and $\CV_{\varsigma}$ of $\CI$ satisfy the assumptions of \cite[Lemma~2.2]{Halasi2006a} (we note that the proof of this result holds for an arbitrary field). In particular, if $\seq{e}{n} \in \CD$ are non-zero orthogonal idempotents such that $\CD = \k e_{1} \oplus \cdots \oplus \k e_{n}$, then $$\dim e_{r}\CV \leq 1\quad \text{and}\quad \dim \CV e_{r} \leq 1$$ for every $1 \leq r \leq n$.

Next, we consider the ideal subgroup $Q = 1+\CL$ of $G$, and the group homomorphism $\map{\varphi_{\varsigma}}{P}{Q^{\circ}}$ (as defined in the previous lemma); we recall that $\ker(\varphi_{\varsigma}) = 1+\CJ_{\varsigma}$. Since $\CL$ is an ideal of $\CA$ with $\dim \CL = \dim \CJ^{n}+1$, we have $\CL = \CJ^{n} \oplus \k u$ where $u = e_{i}ue_{j}$ for some $1 \leq i,j \leq n$; hence $Q = (1+\k u)N$.

Firstly, suppose that $e_{j}\CV = \CV e_{i} = \{0\}$, and let $v \in \CV$ be arbitrary. Then, $uv = ue_{j}v = 0$ and $vu = ve_{i}u = 0$, and so $[1+v,1+\alpha u] = 1$ for all $\alpha \in \k$. It follows that $1+v \in \ker(\varphi_{\varsigma}) = 1+\CJ_{\varsigma}$, and thus $\CV \subseteq \CJ_{\varsigma}$. Therefore, in this case, we have $\CV \subseteq \CI \cap \CJ_{\varsigma} = \CI_{\varsigma}$, and hence $\CI = \CI'_{\varsigma} + \CV \subseteq \CI_{\varsigma}$ which implies that $\CI_{\varsigma} = \CI$ is an ideal of $\CA$.

Now, suppose that $e_{j}\CV \neq \{0\}$, and let $v \in \CV$ be such that $e_{j}\CV = \k v$; since $\CV$ has a $\k$-basis consisting of vectors $w \in \CV$ satisfying $\CD w = w\CD = \k w$, it is clear that $v = ve_{k}$ for some $1 \leq k \leq n$ (see \refl{bimodule}). Then, $\CV = \CV' \oplus \k v$ for some sub-bimodule $\CV'$ of $\CV$; in particular, we have $e_{j}\CV' = \CV'e_{k} = \{0\}$. On the one hand, suppose that $\CV'e_{i} = \{0\}$. Then, the argument above shows that $\CV' \subseteq \CI_{\varsigma}$, and so $\CV' \subseteq \CV \cap \CI_{\varsigma} = \CV_{\varsigma}$. It follows that $\CV' = \{0\}$ (because $\CV' \subsetneq \CV$), and thus $\CV = \k v$. By the definition of $\CV$, we conclude that $\CI = \CI'_{\varsigma}\oplus\k v$, and hence $\dim\CI = \dim\CI'_{\varsigma}+1$. Since $\CI'_{\varsigma} \subseteq \CI_{\varsigma}$ and $\dim \CI_{\varsigma} = \dim \CI -1$, we must have $\CI'_{\varsigma} = \CI_{\varsigma}$, and hence $\CI_{\varsigma}$ is an ideal of $\CA$. On the other hand, let us assume that $\CV'e_{i} \neq \{0\}$, and let $w \in \CV'$ be such that $\CV'e_{i} = \k w$; as above, we must also have $\CD w = \k w$. In this situation, we have $\CV = \CV'' \oplus (\k v \oplus \k w)$ for some sub-bimodule $\CV''$ of $\CV$. Since $e_{j}\CV'' = \CV''e_{i} = \{0\}$, we may repeat the argument above to conclude that $\CV'' \subseteq \CV \cap \CJ_{\varsigma} = \CV_{\varsigma}$. Therefore, $\CV'' = \{0\}$, and so $\CV = \k v \oplus \k w$.

Since $k\neq i$ (otherwise, $v = ve_{i} \in \CV e_{i} = \k w$), we have $vu = 0$, and thus $[1+v,1+u] = 1-u'v$ where $u' \in \CJ$ is such that $(1+u)^{-1} = 1+u'$. Since $u'v \in u\CA v$, we see that $(u'v)^{2} \in (u\CA v)^{2} = \{0\}$, and thus $S = 1+\k(u'v)$ is a $T$-invariant algebra subgroup of $N$; indeed, we have $\CD(u'v)\CD = \k u'v$ (because $u' = e_{i}u'$ and $v = ve_{k}$). Let $\alpha \in \k^{\times}$ be arbitrary, and choose $t \in T$ such that $t^{-1} u' = \alpha u'$ and $vt = v$; notice that, since $i \neq k$, it is enough to choose $t \in T$ satisfying $te_{i} = \alpha^{-1} e_{i}$ and $te_{k} = e_{k}$. It follows that $$[1+v,1+u]^{t} = (1-u'v)^{t} = 1-t^{-1} u'vt = 1-\alpha u'v,$$ and thus the restriction $\varsigma_{S}$ of $\varsigma$ to $S$ is a (smooth) character of $S$ which is constant on $S\setminus\{1\}$ (because $\varsigma$ is $T$-invariant). Therefore, $\varsigma_{S}$ must be the trivial character, and thus $\varsigma([1+v,1+u]) = 1$. It follows that $1+v \in 1+\CJ_{\varsigma}$, and so $v \in \CV \cap \CJ_{\varsigma} = \CV_{\varsigma}$. Since $\k v \neq \CV$ and $\CD v = v \CD = \k v$, we must have $\k v = \{0\}$, a contradiction.

The proof is complete.
\end{proof}

We are now able to prove the following crucial result.

\begin{proposition} \label{extension}
Let $\varsigma \in N^{\circ}$ be $P$-invariant, and let $\CJ_{\varsigma} \subseteq \CJ$ be defined as in \refl{vphi}. 
Then, $[Q,Q] \subseteq \ker(\varsigma)$, and there exists $\vartheta \in Q^{\circ}$ such that $\vartheta_{N} = \varsigma$; moreover, the following properties hold.
\begin{enumerate}
\item $P_{\vartheta'} = 1+\CJ_{\varsigma}$ for all $\vartheta' \in Q^{\circ}$ such that $\vartheta'_{N} = \varsigma$.
\item If $P_{\vartheta} \neq P$ and if $\vartheta' \in Q^{\circ}$ is such that $\vartheta'_{N} = \varsigma$, then there exists $g \in P$ such that $\vartheta' = \vartheta^{g}$.
\end{enumerate}
\end{proposition}

\begin{proof}
We start by observing that $[Q,Q] \subseteq \ker(\varsigma)$. Indeed, since $\dim\CL = \CJ^{n}+1$, there exists $a \in \CL$ such that $\CL = \CJ^{n} \oplus \k a$, and hence $Q = (1+\k a)N$. Since $[1+\alpha a, 1+\beta a] = 1$ for all $\alpha,\beta \in \k$, we see that $\varsigma([1+\k a, 1+\k a]) = \{1\}$, and this clearly implies that $\varsigma([Q,Q]) = \{1\}$ (because $\varsigma$ is $P$-invariant). Let $V$ be an irreducible quotient of the smoothly induced $Q$-module $\Ind^{Q}_{N}(\C_{\varsigma})$; as in the proof of \refl{restriction}, the existence of $V$ follows from \cite[Theorem~1.3]{Boyarchenko2011a} and \cite[Corollary~4.8]{Boyarchenko2011a}. Since $N$ is a normal subgroup of $Q$ and $\varsigma$ is $Q$-invariant, we have $x\phi = \varsigma(x)\phi$ for all $x \in N$ and all $\phi \in \Ind^{Q}_{N}(W)$, and thus $xv = \varsigma(x)v$ for all $x \in N$ and all $v \in V$. Since $[Q,Q] \subseteq \ker(\varsigma)$, it follows from Schur's lemma that $\dim V = 1$, and thus $V$ affords a character $\vartheta \in Q^{\circ}$ which clearly satisfies $\vartheta_{N} = \varsigma$.

In order to prove properties (i) and (ii), we consider the group homomorphism $\map{\varphi_{\varsigma}}{P}{Q^{\circ}}$ as defined in \refl{vphi}; we recall that $\varphi_{\varsigma}(P) \subseteq N^{\perp}$ and that $\ker(\varphi_{\varsigma}) = 1+\CJ_{\varsigma}$ where $\CJ_{\varsigma}$ is an ideal of $\CJ$ satisfying $\CJ^{2} \subseteq \CJ_{\varsigma}$ and $\dim \CJ_{\varsigma} \geq \dim \CJ - 1$. On the one hand, (i) follows because $P_{\vartheta'} = \ker(\varphi_{\varsigma}) = 1+\CJ_{\varsigma}$ for all $\vartheta' \in Q^{\circ}$ such that $\vartheta'_{N} = \varsigma$. On the other hand, let us assume that $P_{\vartheta} \neq P$ (hence, $\ker(\varphi) \neq P$ and $\CJ_{\varsigma} \neq \CJ$), and let $x \in P$ be such that $\varphi(x) \in Q^{\circ}$ is not identically equal to $1$. Let $a \in \CJ$ be such that $x = 1+a$; then, \refeq{scalar} implies that $\varphi_{\varsigma}(1+\alpha a) \in \varphi_{\varsigma}(P) = N^{\perp}$ for all $\alpha \in \k$. Since $N^{\perp} \cong (Q\slash N)^{\circ}$ and $Q\slash N \cong 1+(\CL\slash\CJ^{n}) \cong \k^{+}$, it is straightforward to show that the mapping $\alpha \mapsto \varphi_{\varsigma}(1+\alpha a)$ defines group isomorphism $\k^{+} \cong N^{\perp}$ (we recall that $\k$ is a self-dual field). In particular, it follows that $N^{\perp} = \set{\varphi_{\varsigma}(1+\alpha a)}{\alpha \in \k}$, and so the map $\map{\varphi_{\varsigma}}{P}{N^{\perp}}$ is surjective and $P\slash P_{\vartheta} \cong N^{\perp} \cong \k^{+}$.

To conclude the proof of (ii), let $\vartheta' \in Q^{\circ}$ be such that $\vartheta'_{N} = \varsigma$, and consider the character $\vartheta'\vartheta^{-1} \in Q^{\circ}$. It is obvious that $\vartheta'\vartheta^{-1} \in N^{\perp}$, and thus there exists $\alpha \in \k$ such that $\vartheta'\vartheta^{-1} = \varphi_{\varsigma}(1+\alpha a)$. If we set $g = (1+\alpha a)^{-1}$, then
\begin{align*}
\vartheta'(x)\vartheta(x)^{-1} &= \varsigma([g^{-1},x^{-1}]) = \varsigma(gxg^{-1}x^{-1}) \\ &= \vartheta(gxg^{-1}x^{-1}) = \vartheta(gxg^{-1})\vartheta(x)^{-1},
\end{align*}
and hence $\vartheta'(x) = \vartheta(gxg^{-1})$ for all $x \in Q$, as required.
\end{proof}

\begin{proposition} \label{extension'}
Let $\varsigma \in N^{\circ}$ be $P$-invariant, and let $\vartheta \in Q^{\circ}$ be such that $\vartheta_{N} = \varsigma$. Then, $G_{\vartheta}$ is the unit group of some subalgebra of $\CA$.
\end{proposition}

\begin{proof}
In the case where $\vartheta$ is $P$-invariant, the result follows by \refp{main1}: hence, we assume that $P_{\vartheta} \neq P$. Let $\CD$ be the diagonal subalgebra of $\CA$, and let $T = \CD^{\times}$ be the diagonal subgroup of $G$. By \refp{centraliser1}, we know that $T_{\varsigma}$ is the unit group of some subalgebra $\CD_{\varsigma}$ of $\CD$; similarly, $T_{\vartheta}$ is the unit group of some subalgebra $\CD_{\vartheta}$ of $\CD$.

Since $\varsigma$ is $P$-invariant, we see that the $G_{\varsigma}$ of $\varsigma$ is the unit group of the subalgebra $\CA_{\varsigma} = \CD_{\varsigma} \oplus \CJ$ of $\CA$; indeed, we have $G_{\varsigma} = T_{\varsigma}P$. Let $\CJ_{\varsigma}$ be the ideal of $\CJ$ defined as in \refl{vphi}, and note that $1+\CJ_{\varsigma} = P_{\vartheta}$ is the centraliser of $\vartheta$ in $P$ (by the previous proposition, because $\vartheta_{N} = \varsigma$. Since $\varsigma$ is $G_{\varsigma}$-invariant, \refl{ideal} implies that $\CJ_{\varsigma}$ is an ideal of $\CA(\varsigma)$, and thus $T_{\varsigma} P_{\vartheta}$ is the unit group of the subalgebra $\CB_{\varsigma} = \CD_{\varsigma} \oplus \CJ_{\varsigma}$ of $\CA_{\varsigma}$ (and hence of $\CA$). We also observe that $G_{\vartheta} \subseteq G_{\varsigma}$ (because $\vartheta_{N} = \varsigma$), and that $T_{\vartheta}P_{\vartheta}$ is a normal subgroup of $G_{\vartheta}$ (but not necessarily the unit group of a subalgebra of $\CA$).

Since $\CJ_{\varsigma}$ is an ideal of $\CJ$ with $\dim\CJ_{\varsigma} = \dim\CJ-1$, there exists $a \in \CJ$ such that $\CJ = \CJ_{\varsigma} \oplus \k a$ and $\CD a = a\CD = \k a$ (see \refl{bimodule}). 
Since $G_{\vartheta} \subseteq G_{\varsigma} = T_{\varsigma}P$ and $P_{\vartheta} \subseteq G_{\vartheta}$, every element $g \in G_{\vartheta}$ is uniquely written as a product $g = tx$ for $t \in T_{\varsigma}$ and $x \in 1+\k a$. In fact, for every $t \in T_{\varsigma}$, there is a unique element $x(t) \in 1+\k a$ such that $tx(t) \in G_{\vartheta}$. To see this, let $t \in T_{\varsigma}$ be arbitrary. Then, $\vartheta^{t} \in Q^{\circ}$ satisfies $(\vartheta^{t})_{N} = \varsigma^{t} = \varsigma$, and thus \refp{extension} implies that $\vartheta^{t} = \vartheta^{x}$ for some $x \in P$. Therefore, $\vartheta^{tx^{-1}} = \vartheta$, and hence $tx^{-1} \in G_{\vartheta}$. Since $P = (1+\k a) (1+\CJ_{\varsigma}) = (1+\k a) P_{\vartheta}$ and $P_{\vartheta} \subseteq G_{\vartheta}$, we have $x^{-1} \in x(t)P_{\vartheta}$ for some $x(t) \in 1+\k a$, and so $\vartheta^{tx(t)} = \vartheta^{xx(t)} = \vartheta$; notice that $x(t)$ is uniquely determined by $t \in T_{\varsigma}$.

Suppose that $T_{\varsigma} = T_{\vartheta}$. If this is the case, then $\vartheta^{x(t)} = \vartheta^{tx(t)} = \vartheta$, and thus $x(t) \in G_{\vartheta} \cap P = P_{\vartheta}$ for all $t \in T_{\varsigma}$. By the above, we conclude that $G_{\vartheta} = T_{\varsigma}P_{\vartheta}$ is the unit group of the subalgebra $\CB_{\varsigma}$ of $\CA$. Therefore, we henceforth assume that $T_{\varsigma} \neq T_{\vartheta}$.

For every $t \in T_{\varsigma}$, let $\alpha(t) \in \k$ be such that $x(t) = 1+\alpha(t)a$. It is straightforward to check that the mapping $tT_{\vartheta} \mapsto \alpha(t)$ defines an injective map $$\map{\alpha}{T_{\varsigma}\slash T_{\vartheta}}{\k}.$$ Since $T_{\varsigma} = (\CD_{\varsigma})^{\times}$ of $\CD$, the stabiliser $(T_{\varsigma})_{a} = \set{t \in T_{\varsigma}}{t^{-1} at = a}$ is the unit group of the subalgebra $(\CD_{\varsigma})_{a} = \set{d \in \CD_{\varsigma}}{da = ad}$ of $\CD_{\varsigma}$. Moreover, the mapping $d \mapsto da-ad$ defines a surjective $\k$-linear map $\CD_{\varsigma} \to \k a$ with kernel $(\CD_{\varsigma})_{a}$, and so $\dim (\CD_{\varsigma})_{a} = \dim \CD_{\varsigma}-1$. On the other hand, it is straightforward to check that $\alpha$ induces (by restriction) a group homomorphism $$\map{\widetilde{\alpha}}{\big((T_{\varsigma})_{a} T_{\vartheta}\big)\slash T_{\vartheta}}{\k^{+}}.$$ Since $\big((T_{\varsigma})_{a} T_{\vartheta}\big)\slash T_{\vartheta} \cong (T_{\varsigma})_{a}\slash (T_{\vartheta} \cap (T_{\varsigma})_{a}$ is, either the trivial group, or the direct product of a finite number of copies of the multiplicative group $\k^{\times}$ (because $(T_{\varsigma})_{a}$ and $T_{\vartheta} \cap (T_{\varsigma})_{a}$ are unit groups of subalgebras of $\CD$), we must have $(T_{\varsigma})_{a} = T_{\vartheta} \cap (T_{\varsigma})_{a}$; indeed, if we choose a root of unity $\zeta \in \k^{\times}$ of order coprime to the characteristic of the residue field of $\k$, then we must have $\widetilde{\alpha}(\zeta) = 0$. Therefore, we conclude that $(T_{\varsigma})_{a} \subseteq T_{\vartheta}$, and so $(\CD_{\varsigma})_{a} \subseteq \CD_{\vartheta} \subsetneq \CD_{\varsigma}$. Since $\dim (\CD_{\varsigma})_{a} = \dim \CD_{\varsigma}-1$, it follows that $\CD_{\vartheta} =(\CD_{\varsigma})_{a}$, and thus $T_{\vartheta} = (T_{\varsigma})_{a}$ and $T_{\varsigma}\slash T_{\vartheta} \cong \k^{\times}$.

Since $P_{\vartheta}$ is an ideal subgroup of $G_{\varsigma}$ and $P_{\vartheta} \subseteq G_{\vartheta} \subseteq G_{\varsigma}$, it is also a normal subgroup of $G_{\vartheta}$, and thus the mapping $t \mapsto (tx(t))P_{\vartheta}$ defines a bijection $$\map{\beta}{T_{\varsigma}}{G_{\vartheta}\slash P_{\vartheta}}.$$ Since $P$ is a normal subgroup of $G$, we have $$\big(tt' x(tt')\big) \big(t'x(t')\big)^{-1} \big(tx(t))^{-1} \in P \cap G_{\vartheta} = P_{\vartheta},$$ and so $\beta(tt') = \beta(t)\beta(t')$ for all $t,t' \in T_{\varsigma}$. It follows that $\beta$ is a group isomorphism, and hence $G_{\vartheta}\slash P_{\vartheta}$ is an abelian group. Therefore, we conclude that $\big(T_{\vartheta}P_{\vartheta}\big)\slash P_{\vartheta}$ is a normal subgroup of $G_{\vartheta}\slash P_{\vartheta}$, and thus $T_{\vartheta}P_{\vartheta}$ is a normal subgroup of $G_{\vartheta}$. Since $\beta(T_{\vartheta}) = \big(T_{\vartheta}P_{\vartheta}\big)\slash P_{\vartheta}$, we see that $\beta$ induces naturally a group isomorphism $$\map{\widetilde{\beta}}{T_{\varsigma}\slash T_{\vartheta}}{G_{\vartheta}\slash \big(T_{\vartheta}P_{\vartheta}\big)}.$$

Now, for every $t \in T_{\varsigma}$, we have $tat^{-1} \in \k a$, and hence there is $\lambda(t) \in \k^{\times}$ such that $tat^{-1} = \lambda(t) a$. The mapping $t \mapsto \lambda(t)$ defines a group homomorphism $\map{\lambda}{T_{\varsigma}}{\k^{\times}}$ with $\ker(\lambda) = (T_{\varsigma})_{a} = T_{\vartheta}$. On the other hand, since $\CJ = \CJ_{\varsigma} \oplus \k a$, for every $x \in P$, there exists $\mu(x) \in \k$ such that $x \in (1+\mu(x)a) P_{\vartheta}$, and the mapping $x \to \mu(x)$ defines a group homomorphism $\map{\mu}{P}{\k^{+}}$ with $\ker(\mu) = P_{\vartheta}$. Since every element $g \in T_{\varsigma}P$ is uniquely written as a product $g = tx$ for $t \in T_{\varsigma}$ and $x \in P$, we may define a map $\map{\psi}{T_{\varsigma}P}{\GL_{2}(\k)}$ by the rule $$\psi(tx) = \begin{bmatrix} \lambda(t) & \mu(x) \\ 0 & 1 \end{bmatrix},\qquad t \in T_{\varsigma},\ x \in P.$$ Since $(tx)(t'x') = (tt')((t')^{-1} xt'x')$ and $\mu((t')^{-1} xt'x') = \lambda(t')\mu(x)+\mu(x')$, we see that $\psi((tx)(t'x')) = \psi(tx)\psi(t'x')$ for all $t,t' \in T_{\varsigma}$ and all $x,x' \in P$, which means that $\psi$ is a group homomorphism. It is clear that $\ker(\psi) = T_{\vartheta}P_{\vartheta}$, and so $\psi$ induces a group isomorphism $\map{\widetilde{\psi}}{\big(T_{\varsigma}P\big)\slash \big(T_{\vartheta}P_{\vartheta}\big)}{M_{2}}$ where $M_{2}$ denotes the mirabolic subgroup $$M_{2} = \dset{\begin{bmatrix} r & s \\ 0 & 1 \end{bmatrix}}{r \in \k^{\times},\ s \in \k}$$ of $\GL_{2}(\k)$.

Finally, consider the image $M'_{2} = \widetilde{\psi}\big(G_{\vartheta}\slash \big(T_{\vartheta}P_{\vartheta}\big)\big)$; recall that $G_{\vartheta}$ is a subgroup of $T_{\varsigma}P$. Since there are group isomorphisms $$G_{\vartheta}\slash \big(T_{\vartheta}P_{\vartheta}\big) \cong T_{\varsigma}\slash T_{\vartheta} \cong \k^{\times},$$ we conclude that $M'_{2} \cong \k^{\times}$ is a commutative subgroup of $M_{2}$. For every $t \in T_{\varsigma}$, we have $tx(t) \in G_{\vartheta}$, and $$\psi(tx(t)) = \begin{bmatrix} \lambda(t) & \mu(x(t)) \\ 0 & 1 \end{bmatrix};$$ moreover, notice that the matrix $\psi(tx(t))$ is semisimple for all $t \in T_{\varsigma}\setminus T_{\vartheta}$. Therefore, since $M'_{2}$ is commutative and consists of semisimple matrices, there exists $x \in \GL_{2}(\k)$ such that $$x\begin{bmatrix} \lambda(t) & \mu(x(t)) \\ 0 & 1 \end{bmatrix}x^{-1} = \begin{bmatrix} \lambda(t) & 0 \\ 0 & 1 \end{bmatrix}$$ for all $t \in T_{\varsigma}$; in fact, we may choose $x \in M_{2}$. Let $g \in T_{\varsigma}P$ be such that $\psi(g) = x$. Then, $$\psi(gG_{\vartheta}g^{-1}) = xM'_{2}x^{-1} = \dset{\begin{bmatrix} \lambda(t) & 0 \\ 0 & 1 \end{bmatrix}}{t \in T_{\varsigma}} = \psi(T_{\varsigma}),$$ and thus $gG_{\vartheta}g^{-1} = T_{\varsigma}P_{\vartheta}$ is the unit group of the subalgebra $\CB_{\varsigma}$ of $\CA$. It follows that $G_{\vartheta}$ is the unit group of the subalgebra $g^{-1} \CB_{\varsigma} g$ of $\CA$, and this completes the proof.
\end{proof}

We are now able to prove our main result.

\begin{proof}[Proof of \reft{main}]
We proceed by induction on $\dim \CA$, the result being obvious if $\dim \CA = 1$. Therefore, we assume that $\dim \CA \geq 2$, and that the result is true whenever $\CA'$ is a subalgebra $\CA$ with $\dim\CA' \lneq \dim\CA$.

Let $V$ be an arbitrary irreducible smooth $G$-module, and let $V'$ be an irreducible quotient of $\Res^{G}_{P}(V)$ (the existence of which is guaranteed by \refl{restriction}). In spite of \refp{main1}, we may assume that $\dim V' \geq 2$. In this situation, there is an integer $m \geq 2$ such that $\CJ^{m} \neq \{0\}$ and $\CJ^{m+1} = \{0\}$; notice that $\CJ^{2} \neq \{0\}$ (otherwise, $P = 1+\CJ$ is abelian, and hence $V'$ must be one-dimensional). Since $1+\CJ^{m}$ lies in the centre of $P$, Schur's lemma implies that $1+\CJ^{m}$ acts on $V'$ by scalar multiplications, and thus we may choose the smallest positive integer $n$ for which there exists $\varsigma \in (1+\CJ^{n})^{\circ}$ such that $gv' = \varsigma(g)v'$ for all $g \in 1+\CJ^{n}$ and all $v' \in V'$. We note that, since $V'$ is an irreducible smooth $P$-module with $\dim V' \geq 2$, we must have $n \geq 2$; furthermore, since $[1+\CJ, 1+\CJ^{n-1}] \subseteq 1+\CJ^{n}$, the minimal choice of $n$ implies that $\varsigma$ is not identically equal to $1$ (otherwise, Schur's lemma would imply that the subgroup $1+\CJ^{n-1}$ acts on $V'$ by scalar multiplications). Since $\CJ^{n-1}$ and $\CJ^{n}$ are ideals of $\CA$, \refl{bimodule} implies that $\CJ^{n-1} = \CL_{1} + \cdots + \CL_{t}$ for some ideals $\seq{\CL}{t}$ of $\CA$ satisfying $\CJ^{n} \subseteq \CL_{i} \subseteq \CJ^{n-1}$ and $\dim\big(\CL_{i}\slash\CJ^{n}\big) = 1$ for all $1 \leq i \leq t$. By the minimal choice of $n$, we must have $\varsigma[1+\CJ, 1+\CL_{i}] \neq \{1\}$ for some $1 \leq i \leq t$ (otherwise, we would have $[1+\CJ, 1+\CJ^{n-1}] \subseteq \ker(\varsigma)$, and hence $1+\CJ^{n-1}$ would act on $V'$ by scalar multiplications).


Let $N = 1+\CJ^{n}$, and let $Q = 1+\CL$ where we set $\CL = \CL_{i}$. The argument used in the proof of \refl{restriction} shows the smooth $Q$-module $\Res^{P}_{Q}(V')$ has an irreducible quotient $V''$. Since $[Q,Q] \subseteq \ker(\varsigma)$, Schur's lemma implies that $V''$ is one-dimensional, and thus it affords a character $\vartheta \in Q^{\circ}$. [Notice that the extreme case where $n = 2$ and $\dim\CJ = \dim\CJ^{2}+1$ cannot occur; indeed, in this situation, we must have $Q = P$, and hence $V'' = V'$ which contradicts the assumption $\dim V' \geq 2$.] In particular, we have $V'_{\vartheta} \neq \{0\}$, and thus $V_{\vartheta} \neq \{0\}$ (by \cite[Proposition~2.35]{Bernstein1976a} because $P$ is an $\ell_{c}$-group). By \refl{orbit2}, $V_{\vartheta}$ is an irreducible $G_{\vartheta}$-module and we have $V = \CInd^{G}_{G_{\vartheta}}(V_{\vartheta})$. Since $N$ acts on $V'$ (hence, on $V''$) via the character $\varsigma$, we must have $\vartheta_{N} = \varsigma$, and thus $G_{\vartheta}$ is the unit group of some subalgebra $\CA'$ of $\CA$ (by \refp{extension'}). Since $\vartheta([P,Q]) = \varsigma([P,Q]) \neq \{1\}$, we must have $P_{\vartheta} \neq P$, and thus $G_{\vartheta} \neq G$. Therefore, we have $\dim\CA' \lneq \dim\CA$, and thus it follows by induction that there exists a subalgebra $\CB$ of $\CA'$ such that $V_{\vartheta} \cong \CInd_{H}^{G_{\vartheta}}(W)$ where $H = \CB^{\times}$ is the unit group of $\CB$ and $W$ is a one-dimensional $H$-module. By transitivity of c-induction (see \cite[Proposition~2.25(b)]{Bernstein1976a}), we conclude that $$V \cong \CInd^{G}_{G_{\vartheta}}\Big( \CInd^{G_{\vartheta}}_{H}(W) \Big) \cong \CInd^{G}_{H}(W),$$ as required.
\end{proof}

We conclude the present work with two remarks concerning with admissibility and unitarisability of an arbitrary irreducible smooth $G$-module. On the one hand, we aim to establish the following result. 

\begin{theorem} \label{admissible}
Let $V$ be an irreducible smooth $G$-module, and let $\CB$ be a subalgebra of $\CA$ such that $V \cong \CInd^{G}_{H}(W)$ where $H = \CB^{\times}$ is the unit group of $\CB$ and $W$ is a one-dimensional smooth $H$-module. Then, the following are equivalent.
\begin{enumerate}
\item $H$ contains the diagonal subgroup $T$ of $G$.
\item The smooth $G$-module $V$ is admissible.
\item There is an isomorphism of $G$-modules $V \cong \Ind^{G}_{H}(W)$.	
\end{enumerate}
\end{theorem}

\begin{proof}
For simplicity of reading, we consider several (independent) steps; in the first step, we establish that both (ii) and (iii) imply (i).

\begin{step} \label{admissible1}
Assume that, either $V$ is admissible, or $V \cong \Ind^{G}_{H}(W)$. Then, $T \subseteq H$.
\end{step}

\begin{proof}
Let $\CD'$ be the diagonal subalgebra of $\CB$, and let $T' = (\CD')^{\times}$ be the unit group of $\CD'$ (hence, $T'$ is the diagonal subgroup of $H$). On the other hand, let $\CD$ be the diagonal subalgebra of $\CA$, and note that $\CD$ is clearly a $\CD'$-bimodule. Therefore, \refl{bimodule} implies that $\CD = \CD' \oplus \CD''$ for some sub-bimodule $\CD''$ of $\CD$, and thus $T = T'T''$ where $T'' = (\CD'')^{\times}$ is the unit group of $\CD''$. Then, $G$ decomposes as the semidirect product $G = G'T''$ where $G' = T'P$; moreover, we have $\CB = \CD'\oplus\CJ(\CB)$ where $\CJ(\CB)$ denotes the Jacobson radical of $\CB$, and thus $H \subseteq G'$.

Let $V' = \CInd^{G'}_{H}(W)$, so that $V \cong \CInd^{G}_{H}(W) \cong \CInd^{G}_{G'}(V')$; notice also that $\Ind^{G}_{H}(W) \cong \Ind^{G}_{G'}(V')$. Since $G = G'T''$ is a semidirect product, every element of $G$ is uniquely factorised as a product $gt$ for $g \in G'$ and $t \in T''$, and hence every function $\phi \in \Ind^{G}_{G'}(W')$ is uniquely determined by the rule $$\phi(gt) = g\phi(t),\qquad g \in G',\ t \in T'';$$ in particular, a function $\phi \in \Ind^{G}_{G'}(W')$ lies in $\CInd^{G}_{G'}(W')$) if and only if its restriction to $T''$ has compact support. Since $T'' \cong (\k^{\times})^{r}$ for some nonnegative integer $r \geq 0$, we conclude that $\Ind^{G}_{G'}(V') = \CInd^{G}_{G'}(V')$ if and only if $T'' = \{1\}$ (that is, if and only if $r = 0$); in other words, we have $\Ind^{G}_{G'}(V') = \CInd^{G}_{G'}(V')$ if and only if $G = G'$ which occurs if and only if $T \subseteq H$. 
Finally, suppose that $V \cong \Ind^{G}_{H}(W) \cong \Ind^{G}_{G'}(V')$. Then, $\Ind^{G}_{G'}(V')$ is an irreducible smooth $G$-module, and so $\Ind^{G}_{G'}(V') = \CInd^{G}_{G'}(V')$. By the above, we conclude that $T \subseteq H$, as required.



On the other hand, suppose that $V \cong \CInd^{G}_{H}(W)$ is admissible. Let $\map{\delta_{G}}{G}{\R_{+}^{\times}}$ and $\map{\delta_{H}}{H}{\R_{+}^{\times}}$ be the unimodular characters of $G$ and $H$, respectively; for the definition, we refer to  \cite[Section~3.3]{Bushnell2006a}. If $V^{\vee}$ denotes the smooth dual of $V$ (see \cite[Section~2.8]{Bushnell2006a}), then \cite[Theorem~3.5]{Bushnell2006a} implies that there is a natural isomorphism $$V^{\vee} \cong (\CInd^{G}_{H}(W))^{\vee} \cong \Ind^{G}_{H}(\delta_{G\slash H} \otimes W^{\vee})$$ where $\delta_{G/H} = (\delta_{H})^{-1} (\delta_{G})_{H}$ and where the smooth $H$-module $\delta_{G\slash H} \otimes W^{\vee}$ has underlying vector space equal to $W^{\vee}$ and $H$-action defined by $$h\cdot w^{\vee} = \delta_{G\slash H}(h)(hw^{\vee}), \qquad h \in H,\ w^{\vee} \in W^{\vee}.$$ Since $V$ is admissible, the smooth dual $V^{\vee}$ is irreducible (see \cite[Proposition~2.10]{Bushnell2006a}), and so the smooth $G$-module $\Ind^{G}_{H}(\delta_{G\slash H} \otimes W^{\vee})$ is also irreducible. By the above, we conclude that $T \subseteq H$, and this completes the proof.
\end{proof}

We next prove that (i) implies both (ii) in the particular situation where the subalgebra $\CB$ has codimension one in $\CA$.

\begin{step} \label{admissible2}
Let $\CJ_{0} = \CJ(\CB)$ denote the Jacobson radical of $\CB$, and suppose that $\dim\CJ_{0} = \dim\CJ -1$ where $\CJ = \CJ(\CA)$. Moreover, assume that $T \subseteq H$. Then, the smooth $G$-module $V$ is admissible.
\end{step}

\begin{proof}
Firstly, we note that $\CJ^{2} \subseteq \CJ_{0}$; otherwise, since $\dim\CJ = \dim\CJ_{0}+1$, we must have $\CJ_{0} + \CJ^{2} = \CJ$, and so Nakayama's Lemma (see, for example, \cite[Theorem~2.3]{Matsumura1989a}) implies that $\CJ_{0} = \CJ$. Now, let $P_{0} = 1+\CJ_{0}$, and note that $P_{0}$ is an ideal subgroup of $G$ with $H = P_{0}T$ (because $T \subseteq H$). Let $\tau \in (P_{0})^{\circ}$ be the character of $P_{0}$ afforded by the one-dimensional smooth $P_{0}$-module $\Res^{H}_{P_{0}}(W)$, and consider the irreducible smooth $G_{\tau}$-module $V_{\tau}$; notice that $V \cong \CInd^{G}_{G_{\tau}}(V_{\tau})$. 
Since $H$ is clearly a subgroup of $G_{\tau}$, there are isomorphisms $$V \cong \CInd^{G}_{H}(W) \cong \CInd^{G}_{G_{\tau}}(\CInd^{G_{\tau}}_{H}(W));$$ in particular, we conclude that the smooth $G_{\tau}$-module $\CInd^{G_{\tau}}_{H}(W)$ is irreducible. Since $x\phi = \tau(x)\phi$ for all $x \in P_{0}$ and all $\phi \in \CInd^{G_{\tau}}_{H}(W)$, it follows from \cite[Corollaire~2 au Th\'eor\`eme~3]{Rodier1977a} that $$\CInd^{G_{\tau}}_{H}(W) \cong V_{\tau};$$ note that, as in the proof of \refl{orbit2}, we should consider the closure $\overline{[P,P]}$ of the commutator subgroup of $P$ and replace $G$ by the quotient group $G \slash \overline{[P,P]}$ and $P$ by $P \slash \overline{[P,P]}$ (see also \cite[Theorem~4.13]{Boyarchenko2011a}). Since $xv = \tau(x)v$ for all $x \in P_{0}$ and all $v \in V_{\tau}$, it follows that $\Res^{G_{\tau}}_{T}(V_{\tau})$ is irreducible, and hence $V_{\tau}$ is one-dimensional (by Schur's Lemma); in particular, we conclude that $H = G_{\tau}$ and that $W \cong V_{\tau}$.

On the other hand, since $V \cong \CInd^{G}_{H}(W)$ and since $\CJ^{2} \subseteq \CJ_{0}$, we see that $xv = \tau(x)v$ for all $x \in 1+\CJ^{2}$ and all $v \in V$. Since $\dim V \geq 2$, the proof of \reft{main} guarantees that there exists an ideal $\CL$ of $\CA$ satisfying $\CJ^{2} \subseteq \CL \subseteq \CJ$ and $\dim(\CL\slash \CJ^{2}) = 1$, and such that $$V \cong \CInd^{G}_{G_{\vartheta}}(V_{\vartheta})$$ for some smooth character $\vartheta \in (1+\CL)^{\circ}$. By the construction of $\CL$, it is clear that $\CL \not\subseteq \CJ_{0}$, and thus $\CJ = \CJ_{0}+\CL$ (because $\dim\CJ_{0} = \dim\CJ -1$); furthermore, it follows from \refp{extension} that $\varsigma = \vartheta_{1+\CJ^{2}}$ is a $P$-invariant smooth character of $1+\CJ^{2}$ such that $P_{\vartheta} = 1+\CJ_{\varsigma}$ where $$\CJ_{\varsigma} = \set{a \in \CJ}{\varsigma([1+a,1+u]) = 1 \all u \in \CL}$$ is an ideal of $\CA$ with $\dim\CJ_{\varsigma} = \dim\CJ-1$ (see also \refl{vphi}). 
We claim that $$(G_{\vartheta})_{\tau} = G_{\vartheta} \cap G_{\tau} = T(P_{0}\cap P_{\vartheta});$$ the inclusion $G_{\vartheta} \cap G_{\tau} \subseteq T(P_{\vartheta} \cap P_{0})$ is clear because $G_{\vartheta} \subseteq TP_{\vartheta}$ and $G_{\tau} = H = TP_{0}$. For the reverse inclusion, let $t \in T$ be arbitrary, and recall from \refp{extension} that 
the $P$-orbit $\vartheta^{P} \subseteq (1+\CL)^{\circ}$ of $\vartheta$ consists of all $\vartheta' \in (1+\CL)^{\circ}$ which satisfy $(\vartheta')_{1+\CJ^{2}} = \varsigma$. In particular, since $\vartheta(txt^{-1}) = \tau(txt^{-1}) = \tau(x)$ for all $x \in 1+\CJ^{2}$, we conclude that $\vartheta^{t} = \vartheta^{x}$ for some $x \in P$; indeed, since $\CJ = \CJ_{0} + \CL$, we have $P = P_{0}(1+\CL)$, and thus $\vartheta^{P} = \vartheta^{P_{0}}$. It follows that $\vartheta^{t} = \vartheta^{x}$ for some $x \in P_{0}$, and thus $tx^{-1} \in G_{\vartheta}$. Since $G_{\vartheta} \subseteq TP_{\vartheta}$, we conclude that $x^{-1} \in P_{\vartheta}$, and thus $\vartheta^{t} = \vartheta^{x} = \vartheta$, that is, $t \in G_{\vartheta}$.

Now, let $\sigma \in (P_{\vartheta} \cap P_{0})^{\circ}$ be the restriction of $\tau$ to $P_{\vartheta} \cap P_{0}$, and note that $(G_{\vartheta})_{\sigma} = (G_{\vartheta})_{\tau}$; thus, we have $$V_{\vartheta} \cong \CInd^{G_{\vartheta}}_{(G_{\vartheta})_{\tau}}((V_{\vartheta})_{\sigma}).$$ Since $xv = \sigma(x)v = \tau(x)v$ for all $x \in P_{\vartheta} \cap P_{0}$ and all $v \in (V_{\vartheta})_{\sigma}$, we conclude that every vector subspace of $(V_{\vartheta})_{\sigma}$ is $(P_{\vartheta}\cap P_{0})$-invariant, and this implies that the restriction $\Res^{(G_{\vartheta})_{\tau}}_{T}((V_{\vartheta})_{\sigma})$ is an irreducible smooth $T$-module. Since $T$ is an abelian $\ell$-group, Schur's lemma implies that the smooth $(G_{\vartheta})_{\tau}$-module $(V_{\vartheta})_{\sigma}$ is one-dimensional. Since $G_{\vartheta} = TP_{\vartheta}$ is the unit group of a proper subalgebra of $\CA$, we may use induction to conclude that the smooth $G_{\vartheta}$-module $V_{\vartheta}$ is an admissible. [We note that $(V_{\vartheta})_{\sigma} \cong \Res^{H}_{H_{\varsigma}}(W)$ where $H_{\varsigma} = T(P_{\vartheta} \cap P_{0})$.] 
Since $G_{\vartheta} = TP_{\vartheta}$, we see that $\vartheta^{G} = \vartheta^{P} = \set{\vartheta' \in (1+\CL)^{\circ}}{(\vartheta')_{1+\CJ^{2}} = \varsigma}$ is a closed subset of $(1+\CL')^{\circ}$, and thus it follows from \cite[Th\'eor\`eme~4]{Rodier1977a} that the smooth $G$-module $V \cong \CInd^{G}_{G_{\vartheta}}(V_{\vartheta})$ is admissible, as required.
%
\end{proof}

In the next step we establish that (i) implies (ii).

\begin{step} \label{admissible3}
Assume that $T \subseteq H$. Then, the smooth $G$-module $V$ is admissible.
\end{step}

\begin{proof}
As in the previous proof, we argue  by induction on the dimension of $\CA$, the result being obvious in the case where $\CA$ is one-dimensional; indeed, the result is obvious in the case where $V$ is one-dimensional, which clearly includes the case where the $\k$-algebra $\CA$ is semisimple (because, if this is the case, then $\CA$ must be commutative). Therefore, we may assume that $\dim V \geq 2$; in particular, the Jacobson radical $\CJ = \CJ(\CA)$ of $\CA$ must be non-zero.

Let $\CJ(\CB)$ denote the Jacobson radical of $\CB$. If $\CJ(\CB) + \CJ^{2} = \CJ$, then $\CJ(\CB) = \CJ$ (by Nakayama's Lemma), and thus $H = T(1+\CJ) = G$ and $V = W$ is one-dimensional. It follows that $\CJ(\CB) + \CJ^{2} \subsetneq \CJ$, and so there exists an ideal $\CJ_{0}$ of $\CA$ such that $\CJ(\CB) + \CJ^{2} \subseteq \CJ' \subsetneq \CJ$ and $\dim \CJ' = \dim \CJ - 1$. Let $P' = 1+\CJ'$, and let $G' = TP'$; notice that $P'$ is an ideal subgroup of $G$, and that $G' = (\CA')^{\times}$ is the unit group of the subalgebra $\CA' = \CD \oplus \CJ'$ of $\CA$ where $\CD$ denotes the diagonal subalgebra of $\CA$. Since $H \subseteq T(1+\CJ(\CB)) \subseteq G'$, it is obvious that $$V' = \CInd^{G'}_{H}(W)$$ is an irreducible smooth $G'$-module; indeed, by the transitivity of compact induction, we see that $V \cong \CInd^{G}_{G'}(V')$. Since $\CA'$ is a proper subalgebra of $\CA$, we know by induction that $V'$ is admissible.

If $V'$ is one-dimensional, then it follows from \refs{admissible2} that $V$ is admissible; thus, we may assume that $\dim V' \geq 2$. In this situation, we repeat step-by-step the proof of \reft{main} to construct an ideal subgroup $Q = 1+\CL$ where $\CL \subsetneq \CJ'$ is an ideal of $\CA$ satisfying $\CJ^{n} \subseteq \CL \subseteq \CJ^{n-1}$ and $\dim(\CL\slash \CJ^{n}) = 1$ for a suitable integer $n \geq 2$, and such that $$V' \cong \CInd^{G'}_{G'_{\vartheta}}(V'_{\vartheta})$$ for some smooth character $\vartheta \in Q^{\circ}$. Indeed, the proof of \reft{main} shows that $$V \cong \CInd^{G}_{G_{\vartheta}}(V_{\vartheta});$$ we claim that $V_{\vartheta} = \CInd^{G_{\vartheta}}_{G'_{\vartheta}}\big(V'_{\vartheta}\big)$. On the one hand, we note that there are isomorphisms
\begin{align*}
V &\cong \CInd^{G}_{G'}(V') \cong \CInd^{G}_{G'}\Big(\CInd^{G'}_{G'_{\vartheta}}\big(V'_{\vartheta}\big)\Big) \\ &\cong \CInd^{G}_{G'_{\vartheta}}\big(V'_{\vartheta}\big) \cong \CInd^{G}_{G_{\vartheta}}\Big(\CInd^{G_{\vartheta}}_{G'_{\vartheta}}\big(V'_{\vartheta}\big)\Big),
\end{align*}
and hence $\CInd^{G_{\vartheta}}_{G'_{\vartheta}}\big(V'_{\vartheta}\big)$ is an irreducible smooth $G_{\vartheta}$-module. On the other hand, since $x\phi = \vartheta(x)\phi$ for all $x \in Q$ and all $\phi \in \CInd^{G_{\vartheta}}_{G'_{\vartheta}}\big(V'_{\vartheta}\big)$, it follows from \cite[Corollaire~2 au Th\'eor\`eme~3]{Rodier1977a} that $$\CInd^{G_{\vartheta}}_{G'_{\vartheta}}\big(V'_{\vartheta}\big) \cong V_{\vartheta},$$ which is precisely what we claimed.

Now, we know from by \refp{extension'} that $G'_{\vartheta}$ is the unit group of some subalgebra $\CA'_{\vartheta}$ of $\CA'$; moreover, as in the proof of \reft{main}, we see that there is a subalgebra $\CB'$ of $\CA'_{\vartheta}$ such that $$V'_{\vartheta} \cong \CInd_{H'}^{G'_{\vartheta}}(W')$$ where $H' = (\CB')^{\times}$ is the unit group of $\CB'$ and $W'$ is a one-dimensional smooth $H'$-module. [Notice that $H' \subseteq G'_{\vartheta}$; however, in the general situation, we are not assuming that $H \subseteq G'_{\vartheta}$.] Since $V'$ is admissible, the smooth $G'_{\vartheta}$-module is also admissible, and hence $T \subseteq H'$ (by \refs{admissible1}). Since $$V_{\vartheta} \cong \CInd^{G_{\vartheta}}_{H'}(W')$$ (by the transitivity of c-induction), we conclude that the smooth $G_{\vartheta}$-module $V_{\vartheta}$ is admissible (by induction, because $G_{\vartheta}$ is the unit group of a proper subalgebra of $\CA$). Finally, we recall from \refp{extension} that $\varsigma = \vartheta_{1+\CJ^{n}}$ is a $P$-invariant smooth character of $1+\CJ^{n}$, and that the $P$-orbit $\vartheta^{P} \subseteq Q^{\circ}$ of $\vartheta$ consists of all $\vartheta' \in Q^{\circ}$ which satisfy $(\vartheta')_{1+\CJ^{2}} = \varsigma$; in particular, $\vartheta^{P}$ is a closed subset of $Q^{\circ}$.  
Since $T \subseteq H' \subseteq G_{\vartheta}$, 
it follows that $\vartheta^{G} = \vartheta^{P}$ is a closed subset of $Q^{\circ}$, and thus we conclude that the smooth $G$-module $V \cong \CInd^{G}_{G_{\vartheta}}(V_{\vartheta})$ is admissible (by \cite[Th\'eor\`eme~4]{Rodier1977a}), as required.
%
\end{proof}

Finally, we prove that (i) implies (iii).

\begin{step} \label{admissible4}
Assume that $T \subseteq H$. Then, there is an isomorphism of smooth $G$-modules $V \cong \Ind^{G}_{H}(W)$.
\end{step}

\begin{proof}
Let $\map{\delta_{G}}{G}{\R_{+}^{\times}}$ and $\map{\delta_{H}}{H}{\R_{+}^{\times}}$ be the unimodular characters of $G$ and $H$, respectively. Since $T \subseteq H$, the quotient space $G\slash H$ is naturally homeomorphic to $\CJ\slash \CJ(\CB)$ where $\CJ(\CB)$ denotes the Jacobson radical of $\CB$, and thus the map $\delta_{G/H} = (\delta_{H})^{-1} (\delta_{G})_{H}$ is identically equal to $1$ (because $\CJ\slash \CJ(\CB)$ is an $\ell_{c}$-group). Therefore, \cite[Theorem~3.5]{Bushnell2006a} implies that 
$$V^{\vee} \cong (\CInd^{G}_{H}(W))^{\vee} \cong \Ind^{G}_{H}(W^{\vee}).$$ On the other hand, since $T \subseteq H$, the smooth $G$-module $V$ is admissible (by the previous lemma), and thus its smooth dual $V^{\vee}$ is also irreducible. Since $\CInd^{G}_{H}(W^{\vee})$ is a submodule of $\Ind^{G}_{H}(W^{\vee})$, we conclude that $V^{\vee} \cong \CInd^{G}_{H}(W^{\vee})$, and thus $$(V^{\vee})^{\vee} \cong \big(\CInd^{G}_{H}(W^{\vee})\big)^{\vee} \cong \Ind^{G}_{H}\big((W^{\vee})^{\vee}\big)$$ (again by \cite[Theorem~3.5]{Bushnell2006a}). Since $(W^{\vee})^{\vee} \cong W$ (because $W$ is one-dimensional) and since $(V^{\vee})^{\vee} \cong V$ (by \cite[Proposition2.9]{Bushnell2006a} because $V$ is admissible), we conclude that $V \cong \Ind^{G}_{H}(W)$, as required.
\end{proof}

The proof of \reft{admissible} is now complete.
\end{proof}

We finish the paper with the following result concerning the unitarisability of an arbitrary irreducible smooth $G$-module.

\begin{theorem} \label{unitarisable}
Let $V$ be an irreducible smooth $G$-module. Then, there is an isomorphism of $G$-modules $V \cong V' \otimes V''$ where $V'$ is a unitarisable irreducible smooth $G$-module and  $V''$ is a one-dimensional smooth $G$-module, and where the $G$-action on $V'\otimes V''$ is given by $g(v'\otimes v'') = (gv')\otimes (gv'')$ for all $g \in G$, all $v' \in V'$ and all $v'' \in V''$.
\end{theorem}

\begin{proof}
Let $\CB$ be a subalgebra of $\CA$ such that $V \cong \CInd^{G}_{H}(W)$ where $H = \CB^{\times}$ is the unit group of $\CB$ and $W$ is a one-dimensional smooth $H$-module; moreover, let $\tau \in H^{\circ}$ be the smooth character of $H$ afforded by $W$. Since $\CB$ is a split basic $\k$-algebra, the group $H$ decomposes as the semidirect product $H = SQ$ where $S$ is the unit group of the diagonal algebra of $\CB$ and where $Q$ is the ideal subgroup of $H$ which corresponds to the Jacobson radical of $\CB$; notice that $Q$ is a normal subgroup of $H$, and that $S$ is isomorphic to a finite direct product of copies of $\k^{\times}$.

It is well-known that $\k^{\times} \cong \CO^{\times} \times \mathbb{Z}$ where $\CO$ denotes the ring of algebraic integers of $\k$ (see, for example, \cite[Proposition~5.8]{Narkiewicz2004a}), and thus every smooth character of $\k^{\times}$ can be naturally identified with a product $\alpha \times \beta$ where $\alpha$ is a smooth character of $\CO^{\times}$ and $\beta$ is a smooth character of the infinite cyclic $\mathbb{Z}$. Since $S$ is isomorphic to a finite direct product of copies of $\k^{\times}$, it follows that $S$ decomposes as a direct product $S = S'\times S''$ of a compact subgroup $S'$ and a discrete subgroup $S''$ which is isomorphic to the direct product of a finite number of copies of $\mathbb{Z}$; in particular, every smooth character of $S$ is naturally identified with a product $\sigma' \times \sigma''$ where $\sigma'$ is a smooth character of $S'$ and $\sigma''$ is a smooth character of $S''$. Since every smooth character of a compact group is unitary, we conclude that there exists a smooth character $\sigma \in S^{\circ}$ such that the product $\sigma \tau_{S} \in S^{\circ}$ is an unitary character of $S$.

Let $\sigma^{\ast} \in H^{\circ}$ be defined by $\sigma^{\ast}(sx) = \sigma(s)$ for all $s \in S$ and all $x \in Q$; we note that $$(\sigma^{\ast}\tau)(sx) = (\sigma\tau_{S})(s)\tau(x),\qquad s \in S,\ x \in Q,$$ and thus $\sigma^{\ast}\tau$ is a unitary character of $H$ (because $Q$ is an $\ell_{c}$-group, and hence $\tau_{Q}$ is a unitary character of $Q$; see \cite[Lemma~4.9]{Boyarchenko2011a}). Let $W' \cong \C_{\sigma^{\ast}}$ be a one-dimensional smooth $H$-module which affords the character $\sigma^{\ast}$, and consider the tensor product $W' \otimes W$. We note that $W' \otimes W$ is one-dimensional and affords the unitary character $\sigma^{\ast} \tau \in H^{\circ}$; hence, the smooth $H$-module $W' \otimes W$ is unitarisable. We define $V'$ to be the c-induced smooth $G$-module $$V' = \CInd^{G}_{H}(W'\otimes W),$$ and claim that $V'$ is unitarisable. To see this, we observe that there is a chain of subgroups $H = H_{0} \subseteq H_{1} \subseteq \cdots \subseteq H_{n} = G$ such that $H_{i-1}$ is normal in $H_{i}$ for all $1 \leq i \leq n$; for each $1 \leq i \leq n$, let $\delta_{i}$ denote the modular character of $H_{i}$. By \cite[Corollary~1.5.4]{Deitmar2014a}, we have $(\delta_{i})_{H_{i-1}} = \delta_{i-1}$ for all $1 \leq i \leq n$; in particular, we conclude that $(\delta_{G})_{H} = \delta_{H}$ where $\delta_{G}$ and $\delta_{H}$ are the modular characters of $G$ and $H$, respectively. It follows from \cite[Theorem~1.5.3]{Deitmar2014a} that there exists a non-zero Radon measure $\mu$ on the quotient space $G\slash H$ which is invariant for the action of $G$; we denote by $\langle \cdot, \cdot \rangle$ the $G$-invariant positive definite Hermitian inner product on $W'\otimes W$, and define $$\big\langle \phi, \psi \big\rangle = \int_{G\slash H} \big\langle \phi(g), \psi(g) \big\rangle d\mu,\qquad \phi,\psi \in \CInd^{G}_{H}(W'\otimes W);$$ notice that every function $\phi \in \CInd^{G}_{H}(W'\otimes W)$ defines naturally a function on $G\slash H$ which clearly has compact support. It is straightforward to check that this formula defines a positive definite Hermitian inner product on $W'\otimes W$ on $\CInd^{G}_{H}(W'\otimes W)$ which is $G$-invariant (because the measure $\mu$ is $G$-invariant). Therefore, we conclude that $V' \cong \CInd^{G}_{H}(W'\otimes W)$ is unitarisable, as claimed.

Finally, as in the proof of \reft{admissible} (\refs{admissible1}), we see that there exists a subgroup $T'$ of $T$ such that $T$ decomposes as a direct product $T = ST'$, and thus the smooth character $\sigma \in S^{\circ}$ can be extended to a smooth character $\sigma' \in T^{\circ}$; moreover, since $G$ is the semidirect product $G = TP$ where $P = 1+\CJ(\CA)$, there is a smooth character $\vartheta \in G^{\circ}$ such that $\vartheta(tx) = \sigma'(t)$ for all $t \in T$ and all $x \in P$. Let $W'' \cong \C_{\vartheta}$ be a one-dimensional smooth $H$-module which affords the character $\vartheta$, and consider the tensor product $W'' \otimes \CInd^{G}_{H}(W)$. For every $w'' \in W''$ and every $\phi \in \CInd^{G}_{H}(W)$, we define the function $\map{\psi(w'' \otimes \phi)}{G}{W'' \otimes W}$ by the rule $$\psi(w'' \otimes \phi)(g) = (gw'') \otimes \phi(g)) = \vartheta(g)(w'' \otimes \phi(g)),\qquad g \in G;$$ it is easy to check that the mapping $w'' \otimes \phi \mapsto \psi(w'' \otimes \phi)$ defines an isomorphism of $G$-modules $$W'' \otimes \CInd^{G}_{H}(W) \cong \CInd^{G}_{H}\big(\Res^{G}_{H}(W'') \otimes W\big).$$ Since we clearly have $\Res^{G}_{H}(W'') \cong W'$ (as $H$-modules), we conclude that $$W'' \otimes \CInd^{G}_{H}(W) \cong \CInd^{G}_{H}\big(W' \otimes W\big),$$ and thus $$V \cong \CInd^{G}_{H}(W) \cong V'' \otimes \CInd^{G}_{H}\big(W' \otimes W\big) \cong V' \otimes V''$$ where $V'' = (W'')^{\vee} \cong \C_{\vartheta^{-1}}$. The result follows.
\end{proof}

%

\end{document}